\documentclass[final]{siamart1116}
\usepackage{url}
\usepackage{amssymb}
\usepackage{amsmath}
\usepackage{stmaryrd}
\usepackage{commath}
\makeatletter
\newcommand{\tnorm}{\@ifstar\@tnorms\@tnorm}
\newcommand{\@tnorms}[1]{%
  \left|\mkern-1.5mu\left|\mkern-1.5mu\left|
   #1
  \right|\mkern-1.5mu\right|\mkern-1.5mu\right|
}
\newcommand{\@tnorm}[2][]{%
  \mathopen{#1|\mkern-1.5mu#1|\mkern-1.5mu#1|}
  #2
  \mathclose{#1|\mkern-1.5mu#1|\mkern-1.5mu#1|}
}
\makeatother

\newcommand{\jump}[1]{\llbracket #1 \rrbracket}

\title{Analysis of a hybridized/interface stabilized finite element
  method for the Stokes equations}
\author{Sander Rhebergen\thanks{Department of Applied Mathematics,
    University of Waterloo, Canada (\url{srheberg@uwaterloo.ca}) }
  \and Garth~N.~Wells\thanks{Department of Engineering, University of
    Cambridge, United Kingdom (\url{gnw20@cam.ac.uk})}}
\headers{Analysis of a hybridized method for the Stokes
  equations}{Sander Rhebergen and Garth N. Wells}
\begin{document}
\maketitle
\begin{abstract}
  Stability and error analysis of a hybridized discontinuous Galerkin
  finite element method for Stokes equations is presented.  The method
  is locally conservative, and for particular choices of spaces the
  velocity field is point-wise solenoidal.  It is shown that the
  method is inf-sup stable for both equal-order and locally
  Taylor--Hood type spaces, and \emph{a priori} error estimates are
  developed.  The considered method can be constructed to have the
  same global algebraic structure as a conforming Galerkin method,
  unlike standard discontinuous Galerkin methods that have greater
  number of degrees of freedom than conforming Galerkin methods on a
  given mesh.  We assert that this method is amongst the simplest and
  most flexible finite element approaches for Stokes flow that provide
  local mass conservation. With this contribution the mathematical
  basis is established, and this supports the performance of the
  method that has been observed experimentally in other works.
\end{abstract}
\section{Introduction}
\label{sec:introduction}

We present analysis of a type of hybridized discontinuous Galerkin
method for the Stokes equations.  These methods can be constructed to
have properties usually associated with discontinuous Galerkin finite
element methods, while retaining the attractive features of continuous
finite element methods, such as reduced discrete problem size.  The
analysis includes the method known as the Interface Stabilized Finite
Element Method (IS-FEM)~\cite{labeur:2007, Labeur:2012,
  Wells:2011} or Embedded Discontinuous Galerkin (EDG)
method~\cite{Cockburn:2009} in the literature.  The interface
stabilized formulation for the incompressible Navier--Stokes equations
with continuous pressure fields was presented in \cite{labeur:2007},
and generalised in \cite{Labeur:2012} for discontinuous pressure
fields.  The formulation is closely related to that of the
hybridizable discontinuous Galerkin method using interior penalty
numerical traces, the so-called IP-H~methods~\cite{Cockburn:2009a}.

The formulation we consider has been shown previously to have a number
of appealing properties when applied to the Stokes
equations~\cite{labeur:2007, Labeur:2012}. These include local mass
conservation (point-wise in cases), experimentally observed optimal
rates of convergence, they can be constructed to have an algorithmic
structure that is identical to a conforming finite element method, and
natural incorporation of stabilizing numerical fluxes when including
advective transport.

Other hybridized discontinuous Galerkin methods for the Stokes
equations have been developed, for example in~\cite{Cockburn:2011,
  Cockburn:2014b}. The method discussed in this work differs from
those in the aforementioned references in the following two
respects. Firstly, the methods considered in~\cite{Cockburn:2011,
  Cockburn:2014b} solve the Stokes equations in the
gradient-velocity-pressure format, resulting in so-called LDG-H type
methods, whereas we consider here the Stokes equations in the
velocity-pressure format. The velocity field resulting from the
formulations in \cite{Cockburn:2011, Cockburn:2014b} is not point-wise
solenoidal, but it is possible to devise post-processing operators for
LDG-H methods that generate approximate velocity fields that are
exactly divergence-free and $H({\rm div})$-conforming. Furthermore, it
can be shown that the post-processed velocity fields converge with
order $k+2$ when using polynomial approximations of order~$k$. The
second difference is that in~\cite{Cockburn:2011, Cockburn:2014b} the
velocity traces are the only global unknowns. Both velocity and
pressure traces are the global unknowns in our formulation (although
pressure traces can be eliminated in some cases).

Analysis of the framework considered in this work, but applied to the
scalar advection-diffusion equation, was presented
in~\cite{Wells:2011}.  Here, we address the Stokes problem and
underpin previous numerical investigations~\cite{Labeur:2012} by
proving stability of the formulation for the Stokes problem and
providing \emph{a priori} error estimates.  A particular motivation
for this work is an observation that block preconditioners, with
multigrid preconditioners applied to the blocks, can be very effective
for the IS-FEM/EDG type formulation.

The remainder of this work is structured as follows. In
\cref{sec:stokes} we present the Stokes equations, followed in
\cref{s:isfem} by the definition of the method we analyse with a
summary of its key features. Stability and boundedness of the method
are shown in \cref{s:conStabBound}, and error estimates are presented
in \cref{s:errorAnalysis}. Conclusions are drawn in
\cref{sec:conclusions}.

\section{The Stokes problem}
\label{sec:stokes}

Let $\Omega\subset\mathbb{R}^d$ be a polygonal ($d = 2$) or polyhedral
($d = 3$) domain, with the boundary of $\Omega$ denoted
by~$\Gamma$. We consider the Stokes problem of finding the velocity
field $u: \Omega \to \mathbb{R}^d$ and the pressure $p: \Omega \to
\mathbb{R}$ such that
\begin{subequations}
  \begin{align}
    \label{eq:stokes_a}
    -\Delta u + \nabla p &= f & & \mbox{in} \ \Omega,
    \\
    \label{eq:stokes_b}
    \nabla\cdot u &= 0 & & \mbox{in}\ \Omega,
    \\
    \label{eq:stokes_c}
    u &= 0 & & \mbox{on}\ \Gamma,
    \\
    \label{eq:stokes_d}
    \int_{\Omega} p \dif x &= 0,
  \end{align}
  \label{eq:stokes}
\end{subequations}
where $f : \Omega \to \mathbb{R}^d$ is the prescribed body force.

For a body force $f \in [L^2(\Omega)]^d$, the weak formulation of the
Stokes problem is given by: find $(u, p) \in [H_0^1(\Omega)]^d \times
L_{0}^{2}(\Omega)$ such that
\begin{subequations}
  \label{eq:stokes_wf_continuous}
  \begin{align}
    a(u,v) + b(v,p) &= \int_{\Omega} f \cdot v \dif x & &
    \forall v\in [H_0^1(\Omega)]^{d},
    \\
    -b(u,q) &=  0 & & \forall q\in L_0^2(\Omega),
  \end{align}
\end{subequations}
where $L^{2}_{0}(\Omega)$ is the space of $L^{2}$ functions on
$\Omega$ with zero mean, and the forms $a$ and $b$ are defined as
\begin{align}
  a(u, v) &:= \int_{\Omega}\nabla u :\nabla v \dif x
  \\
  b(p, v) &:= -\int_{\Omega}p \nabla \cdot v \dif x.
\end{align}

Given that the form $a$ is coercive on $[H_0^1(\Omega)]^{d}$, the
inf-sup condition on $b$ for the Stokes
problem~\cref{eq:stokes_wf_continuous} to be well-posed
is~\cite[Section~4.2.2]{Boffi:book}:
\begin{equation}
  \label{eq:infsupcondition_continuous}
  \beta_c \norm{q}_{0, \Omega} \le \sup_{w \in [H_0^1(\Omega)]^2}
  \frac{b(q, w)}{\norm{w}_{1,\Omega}} \qquad
  \forall q\in L_{0}^{2}(\Omega),
\end{equation}
where $\beta_c > 0$ is a constant depending only on~$\Omega$.

\section{Hybridized discontinuous Galerkin method}
\label{s:isfem}

The method that will be analysed is presented in this section, along
with some of its key conservation properties.

\subsection{Preliminaries}

Let $\mathcal{T} := \{K\}$ be a triangulation of the domain $\Omega$
into non-overlapping cells~$K$.  The characteristic length of a cell
$K$ is denoted by~$h_{K}$.  On the boundary of a cell, $\partial K$,
we denote the outward unit normal vector by~$n$. An interior facet $F$
is shared by two adjacent cells $K^+$ and $K^-$, $F :=
\overline{\partial K^+} \cap \overline{\partial K^-}$ and a boundary
facet is a facet of $\overline{\partial K}$ that lies on $\Gamma$.
The set of all facets is denoted by $\mathcal{F} = \{ F \}$, and the
union of all facets is denoted by~$\Gamma^0$.

We will work with the following finite element function spaces
on~$\Omega$:
\begin{equation}
  \begin{split}
    V_h  &= \cbr{v_h\in \sbr{L^2(\Omega)}^d
        : \ v_h \in \sbr{P_k(K)}^d, \ \forall\ K\in\mathcal{T}},
    \\
    Q_h &= \cbr{q_h\in L^2(\Omega) : \ q_h \in P_m(K) ,\
    \forall \ K \in \mathcal{T}},
  \end{split}
  \label{eqn:spaces_cell}
\end{equation}
and the following finite element spaces on the facets of the
triangulation of~$\Omega$:
\begin{equation}
  \begin{split}
    \bar{V}_h &= \cbr{\bar{v}_h \in \sbr{L^2(\Gamma^0)}^d:\ \bar{v}_h \in
      \sbr{P_{k}(F)}^d\ \forall\ F \in \mathcal{F},\ \bar{v}_h
      = 0 \ \mbox{on}\ \Gamma},
    \\
    \bar{Q}_h &= \cbr{\bar{q}_h \in L^2(\Gamma^0) : \ \bar{q}_h \in
      P_{k}(F) \ \forall\ F \in \mathcal{F}},
  \end{split}
  \label{eqn:spaces_facet}
\end{equation}
where $P_{k}(D)$ denotes the space of polynomials of degree~$k$ on
domain~$D$, with $k \ge 1$ and~$m \le k$.

For notational purposes, we introduce the spaces $V_h^{\star} =
V_h\times \bar{V}_h$, $Q_h^{\star} = Q_h\times \bar{Q}_h$, and
$X_h^{\star} = V_h^{\star}\times Q_h^{\star}$.  Function pairs in
$V_h^{\star}$ and $Q_h^{\star}$ will be denoted by boldface, e.g.,
${\bf v}_h = (v_h, \bar{v}_h)\in V_h^{\star}$ and ${\bf q}_h = (q_h,
\bar{q}_h)\in Q_h^{\star}$.

\subsection{Weak formulation}

We consider the weak formulation as presented in
\cite{Labeur:2012}. For the Stokes problem, it seeks $({\bf u}_{h},
     {\bf p}_{h}) \in X_{h}^{\star}$ such that
\begin{multline}
  \label{eq:momentum_equation_wf}
  \sum_{K\in\mathcal{T}}\int_{K}\nabla u_h : \nabla v_h\dif x
  + \sum_{K\in\mathcal{T}}\int_{\partial K}(\bar{u}_h- u_h)\cdot \frac{\partial v_h}{\partial n}\dif s
  - \sum_{K\in\mathcal{T}}\int_{K}p_h\nabla\cdot v_h\dif x
  \\
  + \sum_{K\in\mathcal{T}}\int_{\partial K}\hat{\sigma}_h n\cdot( v_h-\bar{v}_h)\dif s
  = \sum_{K\in\mathcal{T}}\int_{K} f\cdot v_h\dif x \qquad \forall {\bf v}_h\in V_h^{\star},
\end{multline}
and
\begin{equation}
  \label{eq:massequation_wf}
  \sum_{K\in\mathcal{T}}\int_{K} u_h\cdot\nabla q_h\dif x
  + \sum_{K\in\mathcal{T}}\int_{\partial K}\hat{u}_h\cdot n (\bar{q}_h-q_h)\dif s
  - \int_{\Gamma}\bar{u}_h\cdot n\bar{q}_h\dif s = 0 \qquad
  \forall {\bf q}_h\in Q^{\star}_h,
\end{equation}
with the numerical fluxes $\hat{\sigma}_h$ and $\hat{u}_h$ given by
\begin{equation}
  \label{eq:numericalfluxes}
  \hat{\sigma}_h := -\nabla u_h + \bar{p}_h I
    - \frac{\alpha_{v}}{h_K}(\bar{u}_h - u_h)\otimes n,
  \qquad
  \hat{u}_h := u_h - \alpha_{p} h_K(\bar{p}_h - p_h) n,
\end{equation}
and where $\alpha_{v} > 0$ and $\alpha_{p} \ge 0$ are penalty
parameters. Note that the numerical fluxes can take on different
values on opposite sides of a facet.  We will prove that the
formulation is stable for sufficiently large $\alpha_{v}$, akin to the
standard interior penalty method~\cite{Arnold:2002}.  For a
mixed-order formulation with $m = k - 1$ we will show that
$\alpha_{p}$ can be set to zero, and for the equal-order case ($k =
m$) we will show that $\alpha_{p}$ must be positive.

It will be convenient to express the method in a compact form,
therefore we introduce the bilinear forms:
\begin{subequations}
  \begin{align}
    \label{eq:formA}
    a_h({\bf u}, {\bf v}) :=&
    \sum_{K\in\mathcal{T}}\int_{K}\nabla u : \nabla v \dif x
    + \sum_{K\in\mathcal{T}}\int_{\partial K}\frac{\alpha_{v}}{h_K}(u - \bar{u})\cdot(v - \bar{v}) \dif s
    \\
    \nonumber
    &- \sum_{K\in\mathcal{T}}\int_{\partial K}\sbr{(u-\bar{u})\cdot \frac{\partial v}{\partial n}
    + \frac{\partial u}{\partial n}\cdot(v-\bar{v})} \dif s,
    \\
    \label{eq:formB}
    b_h({\bf p}, {\bf v}) :=&
    - \sum_{K\in\mathcal{T}}\int_{K}p \nabla \cdot v \dif x
    + \sum_{K\in\mathcal{T}}\int_{\partial K}(v-\bar{v})\cdot n
    \bar{p} \dif s,
    \\
    \label{eq:formC}
    c_h({\bf p}, {\bf q}) :=&
    \sum_{K\in\mathcal{T}}\int_{\partial K}\alpha_{p} h_K(p-\bar{p})(q-\bar{q}) \dif s.
  \end{align}
\end{subequations}
Solutions $({\bf u}_h,{\bf p}_h)\in X_h^{\star}$ satisfy
\begin{equation}
  \label{eq:compact_wf}
  B_h(({\bf u}_h, {\bf p}_h), ({\bf v}_h, {\bf q}_h))
  = \sum_{K\in\mathcal{T}}\int_{K}f\cdot v_h \dif x
  \qquad \forall ({\bf v}_h,{\bf q}_h)\in X_h^{\star},
\end{equation}
where
\begin{equation}
  \label{eq:bilinearformStokes}
  B_h(({\bf u}_h, {\bf p}_h), ({\bf v}_h, {\bf q}_h))
  :=
  a_h({\bf u}_h, {\bf v}_h)
  + b_h({\bf p}_h, {\bf v}_h)
  - b_h({\bf q}_h, {\bf u}_h)
  + c_h({\bf p}_h, {\bf q}_h).
\end{equation}

To provide some insights into the method, setting $\Bar{v}_{h} = 0$ we
note that~\cref{eq:momentum_equation_wf} is a cell-wise statement of
the momentum balance, subject to weak satisfaction of the boundary
condition provided by~$\bar{u}_{h}$ (using Nitsche's method).  Setting
$v_{h} = 0$, we note that \cref{eq:momentum_equation_wf} imposes weak
continuity of the numerical flux $\Hat{\sigma}_h$ across facets.
Equation~\cref{eq:massequation_wf} can be interpreted similarly, with
it enforcing the continuity equation locally (in terms of the
numerical flux $\Hat{u}_{h}$) and weak continuity of $\Hat{u}_{h}$
across facets. Different from conventional discontinuous Galerkin
methods, functions on cells are not directly coupled to their
neighbours via the numerical flux. Rather, functions on cells are
coupled indirectly via the `bar' functions that live only on facets.
This has the important implementation consequence that
degrees-of-freedom associated with $u_{h}$ and $p_{h}$ can be
eliminated cell-wise in favour of degrees-of-freedom associated with
$\Bar{u}_{h}$ and~$\Bar{p}_{h}$. This process is commonly known as
\emph{static condensation}. This avoids the greater number of global
degrees-of-freedom associated with standard discontinuous Galerkin
methods compared to conforming methods on the same mesh.

\subsection{Mass and momentum conservation}

It is straightforward to show that the method conserves mass locally
(cell-wise) in terms of the numerical flux~$\Hat{u}_{h}$. Setting
$v_h = \bar{v}_h = 0$ and $\bar{q}_h = 0$, and $q_h = 1$ on a cell $K$
and $q_h = 0$ on $\mathcal{T} \backslash K$
in~\cref{eq:massequation_wf},
\begin{equation}
  \int_{\partial K}\hat{u}_h\cdot n\dif s = 0 \qquad \forall K \in \mathcal{T}.
\end{equation}
In the case that $\alpha_{p} = 0$, $\Hat{u}_{h}$ and $u_{h}$ coincide.
Setting $v_h = \bar{v}_h = 0$ and $q_h = \bar{q}_h = 1$
in~\cref{eq:massequation_wf},
\begin{equation}
  \int_{\Gamma}\bar{u}_h\cdot n\dif s = 0.
\end{equation}

Noteworthy is that for simplices with $\alpha_{p} = 0$ and $m = k -
1$, i.e.~the divergence of a function in the velocity space $V_{h}$ is
contained in the pressure space~$Q_{h}$, the velocity field $u_{h}$ is
point-wise solenoidal within a cell.

Momentum conservation is addressed in \cite{Labeur:2012} for the
incompressible Navier--Stokes equations, where local momentum
conservation in terms of the numerical flux $\hat{\sigma}_h$ was
shown.  To see this, set $v_h = e_j$ on~$K$, where $e_j$ is a
canonical unit basis vector, and $v_h = 0$ on $\mathcal{T} \backslash
K$, $\bar{v}_h = 0$, $q_h = 0$ and $\bar{q}_h = 0$
in~\cref{eq:momentum_equation_wf}. This yields:
\begin{equation}
  \int_{\partial K} \hat{\sigma}_h n \dif x = \int_K f\dif x \qquad
  \forall K \in \mathcal{T}.
\end{equation}

\subsection{Relationship to a $H(\rm div)$-conforming formulation}

For a mixed-order formulation with $m=k-1$, the `bar' function
$\bar{p}_h$ acts as a Lagrange multiplier enforcing continuity of the
normal component of $u_h \in \Bar{V}_h$ across inter-element
boundaries. It is easy to see that by setting ${\bf v}_{h} = {\bf 0}$
and $q_{h} = 0$ in \cref{eq:compact_wf} that the normal component of
the velocity is continuous across facets,
i.e.~$u_{h} \in V_h^{\rm BDM}$, where $V_h^{\rm BDM}$ is a
Brezzi--Douglas--Marini (BDM) finite element space \cite{Boffi:book}:
\begin{equation}
  \label{eq:bdm_space}
  \begin{split}
    V_h^{\rm BDM}(K)
    =& \cbr{ v_h \in \sbr{P_k(K)}^d : v_h\cdot n \in L^2(\partial K),\
      v_h\cdot n|_F\in P_k(F)},
    \\
    V_h^{\rm BDM}
    =& \cbr{v_h \in H({\rm div};\Omega) :\ v_h|_K \in V_h^{\rm BDM}(K),\
      \forall K \in \mathcal{T}}.
  \end{split}
\end{equation}

Defining $V_h^{\star{\rm BDM}} = V_h^{\rm BDM} \times \bar{V}_h$ and
$X_h^{\star{\rm BDM}} = V_h^{\star{\rm BDM}} \times Q_h$, the
formulation in \cref{eq:compact_wf} is a hybridized \cite{Brezzi:1985}
form of: find the $({\bf u}_h, p_h) \in V_h^{\star{\rm BDM}} \times
Q_h$ such that
\begin{equation}
  \label{eq:bdm_compact_wf}
  B_h^{\rm BDM}(({\bf u}_h, p_h), ({\bf v}_h, q_h))
  = \sum_{K\in\mathcal{T}}\int_{K}f\cdot v_h \dif x
  \qquad \forall ({\bf v}_h, q_h)\in X_h^{\star{\rm BDM}},
\end{equation}
where
\begin{equation}
  \label{eq:bdm_bilinearformStokes}
  B_h^{\rm BDM}(({\bf u}_h, p_h), ({\bf v}_h, q_h))
  :=
  a_h({\bf u}_h, {\bf v}_h)
  + b_h^{\rm BDM}(p_h, {\bf v}_h)
  - b_h^{\rm BDM}(q_h, {\bf u}_h),
\end{equation}
and where
\begin{equation}
  \label{eq:bdm_bhform}
  b_h^{\rm BDM}(p, {\bf v}) := -\sum_{K\in\mathcal{T}}\int_{K}p \nabla \cdot v \dif x.
\end{equation}
If \cref{eq:compact_wf} and \cref{eq:bdm_compact_wf} have unique
solutions, then then solution pair $({\bf u}_h,p_h)$ to
\cref{eq:compact_wf} is the solution of \cref{eq:bdm_compact_wf}.

The $H(\rm div)$-conforming formulation will be convenient for
subsequent analysis as we will be able to neglect the Lagrange
multiplier terms. For implementation we recommend
\cref{eq:compact_wf}.

\section{Consistency, stability and boundedness}
\label{s:conStabBound}

In this section we demonstrate consistency, stability, boundedness and
well-posedness. To do this, we introduce extended function spaces
on~$\Omega$:
\begin{align}
  V(h) &:= V_h + \sbr{H_0^1(\Omega)}^d \cap \sbr{H^2(\Omega)}^d,
  \\
  Q(h) &:= Q_h + L_0^2(\Omega) \cap H^1(\Omega),
\end{align}
and extended function spaces on $\Gamma^{0}$ (facets):
\begin{align}
  \bar{V}(h) &:= \bar{V}_h + [H_0^{3/2}(\Gamma^0)]^d,
  \\
  \bar{Q}(h) &:= \bar{Q}_h + H^{1/2}_0(\Gamma^0),
\end{align}
where $[H_0^{3/2}(\Gamma^0)]^d$ and $H_0^{1/2}(\Gamma^0)$ are,
respectively, the trace spaces of $\sbr{H_0^1(\Omega)}^d \cap
\sbr{H^2(\Omega)}^d$ and $L_0^2(\Omega) \cap H^1(\Omega)$ on
facets~$\Gamma^0$. We introduce the trace operator $\gamma :
H^s(\Omega)\to H^{s - 1/2}(\Gamma^0)$ to restrict functions in
$H^s(\Omega)$ to $\Gamma^0$. For functions in $[H^s(\Omega)]^d$ the
trace operator is applied component-wise.  Even when not strictly
necessary, we will use the trace operator to make clear when a
function, usually the exact solution, is being restricted to facets.
For notational purposes we also introduce $V^{\star}(h) := V(h)\times
\bar{V}(h)$, $Q^{\star}(h) := Q(h) \times \bar{Q}(h)$ and
$X^{\star}(h) := V^{\star}(h) \times Q^{\star}(h)$.

We define two norms on $V^{\star}(h)$, namely,
\begin{equation}
\label{eq:stability_norm}
  \tnorm{ {\bf v} }_v^2 := \sum_{K\in\mathcal{T}}\norm{\nabla v }^2_{0,K}
  + \sum_{K\in\mathcal{T}} \frac{\alpha_{v}}{h_K}\norm{\bar{v} - v}^2_{0,\partial K},
\end{equation}
which will be used to prove stability of~$a_h$, and
\begin{equation}
  \label{eq:boundedness_norm}
  \tnorm{ {\bf v} }_{v'}^2 := \tnorm{ {\bf v} }_v^2
  + \sum_{K\in\mathcal{T}}\frac{h_K}{\alpha_v}\norm{\frac{\partial v}{\partial n}}^2_{0,\partial K},
\end{equation}
which will be used to prove boundedness of $a_h$.  From the discrete
trace inequality~\cite[Remark~1.47]{Pietro:book},
\begin{equation}
  \label{eq:trace_inequality}
  h_K^{1/2}\norm{v_h}_{0,\partial K}\le C_t\norm{v_h}_{0,K}
  \qquad \forall v_h \in P_k(K),
\end{equation}
where $C_{t}$ depends on $k$, spatial dimension and cell shape, it
follows that the norms $\tnorm{\cdot}_v$ and $\tnorm{\cdot}_{v'}$ are
equivalent on~$V_h^{\star}$:
\begin{equation}
  \label{eq:equivalentNorms_v_vprime}
  \tnorm{ {\bf v}_h }_{v} \le \tnorm{ {\bf v}_h }_{v'} \le c(1+\alpha_v^{-1})\tnorm{ {\bf v}_h }_{v},
\end{equation}
with $c>0$ a constant independent of $h$,
see~\cite[Eq. (5.5)]{Wells:2011}.

We introduce a `pressure semi-norm':
\begin{equation}
  \label{eq:pressureseminorm}
  \envert{{\bf q}}_p^2 :=
  \sum_{K \in \mathcal{T}} \alpha_p h_K \norm{\bar{q} - q}^2_{0, \partial K},
\end{equation}
and define a norm on~$X_h^{\star}$ by:
\begin{equation}
  \tnorm{({\bf v}_h,{\bf q}_h)}_{v,p}^2 := \tnorm{{\bf v}_h}_v^2
  + \norm{q_h}^2_{0,\Omega} + \envert{{\bf q}_h}^2_p,
\end{equation}
and on $X^{\star}(h)$ we define
\begin{equation}
  \begin{split}
    \tnorm{({\bf v},{\bf q})}_{v',p'}^2
    &:= \tnorm{({\bf v},{\bf q})}_{v,p}^2
    + \sum_{K\in\mathcal{T}}\frac{h_K}{\alpha_v}\norm{\frac{\partial v}{\partial n}}^2_{0,\partial K}
    + \sum_{K\in\mathcal{T}}h_K\norm{\bar{q}}^2_{0,\partial K}
    \\
    &= \tnorm{{\bf v}}_v^2 + \norm{q}^2_{0,\Omega}
    + \envert{{\bf q}}^2_p
    + \sum_{K\in\mathcal{T}}\frac{h_K}{\alpha_v}\norm{\frac{\partial v}{\partial n}}^2_{0,\partial K}
    + \sum_{K\in\mathcal{T}}h_K\norm{\bar{q}}^2_{0,\partial K}
    \\
    &= \tnorm{{\bf v}}_{v'}^2 + \norm{q}^2_{0,\Omega}
    + \envert{{\bf q}}^2_p
    + \sum_{K\in\mathcal{T}}h_K\norm{\bar{q}}^2_{0,\partial K}.
  \end{split}
\end{equation}
Note that \cref{eq:pressureseminorm} vanishes for the case
of~$\alpha_{p} = 0$.

\subsection{Consistency}

We now prove consistency of the method. It is assumed that
$(u, p) \in X$ solves the Stokes problem~\cref{eq:stokes}, where
\begin{equation}
  X := \del{\sbr{H_0^1(\Omega)}^d \cap \sbr{H^2(\Omega)}^d} \times
  \del{L_0^2(\Omega) \cap H^1(\Omega)}.
\end{equation}

\begin{lemma}[Consistency] \label{lem:consistency}
  If $(u, p)\in X$ solves the Stokes problem~\cref{eq:stokes},
  letting ${\bf u} = (u, \gamma(u))$ and ${\bf p} = (p, \gamma(p))$,
  then
  \begin{equation}
    B_h(({\bf u}, {\bf p}), ({\bf v}_h, {\bf q}_h))
    =
    \sum_{K\in\mathcal{T}}\int_Kf\cdot v_h\dif x \qquad
    \forall ({\bf v}_h,{\bf q}_h)\in X_h^{\star}.
  \end{equation}
\end{lemma}
\begin{proof}
  We consider each form in the definition of $B_h$ separately. Using
  that $\Bar{u} = \gamma(u)$ and applying integration by parts
  to~\cref{eq:formA}, we find that
  \begin{equation}
    \label{eq:proof_consis_a}
    \begin{split}
      a_h({\bf u}, {\bf v}_h) = &
      \int_{\Omega}\nabla u : \nabla v_h\dif x
      + \sum_{K\in\mathcal{T}}\int_{\partial K}\frac{\alpha_{v}}{h_K}(\bar{u}-u)\cdot(\bar{v}_h- v_h)\dif s
      \\
      &+ \sum_{K\in\mathcal{T}}\int_{\partial K}\sbr{\frac{\partial u}{\partial n}\cdot(\bar{v}_h-v_h)
        + (\bar{u}- u)\cdot\frac{\partial v_h}{\partial n}}\dif s
      \\
      = &
      \int_{\Omega}\nabla u : \nabla v_h\dif x
      - \sum_{K\in\mathcal{T}}\int_{\partial K}\frac{\partial u}{\partial n} \cdot(v_h-\bar{v}_h)\dif s
      \\
      = &
      -\int_{\Omega} \Delta u \cdot v_h\dif x
      + \sum_{K\in\mathcal{T}}\int_{\partial K}\frac{\partial u}{\partial n}\cdot\bar{v}_h\dif s.
    \end{split}
  \end{equation}
  Next, using that $\bar{p} = \gamma(p)$ and applying integration by
  parts to~\cref{eq:formB}, we find that
  \begin{equation}
    \label{eq:proof_consis_b_pv}
    \begin{split}
      b_h({\bf p}, {\bf v}_h) = &
    - \int_{\Omega}p \nabla \cdot v_h\dif x
    + \sum_{K\in\mathcal{T}}\int_{\partial K}(v_h-\bar{v}_h)\cdot n
    \bar{p}\dif s
    \\
    = &
    \int_{\Omega}\nabla p \cdot v_h\dif x
    + \sum_{K\in\mathcal{T}}\int_{\partial K} v_h\cdot n(\bar{p}-p)\dif s
    - \sum_{K\in\mathcal{T}}\int_{\partial K}\bar{v}_h\cdot n\bar{p}\dif s
    \\
    = &
    \int_{\Omega}\nabla p \cdot v_h\dif x
    - \sum_{K\in\mathcal{T}}\int_{\partial K}\bar{v}_h\cdot n\bar{p}\dif s.
    \end{split}
  \end{equation}
  Adding~\cref{eq:proof_consis_a} and~\cref{eq:proof_consis_b_pv}
  and using that $\bar{p} = \gamma(p)$, we obtain
  \begin{equation}
    \int_{\Omega}\del{-\Delta u + \nabla p}\cdot v_h\dif x
    - \sum_{K\in\mathcal{T}}\int_{\partial K} \del{-\nabla u + \bar{p}I} n
      \cdot \bar{v}_h \dif s
      = \int_{\Omega} f \cdot v_h \dif x.
  \end{equation}
  Consider the facet integrals:
  \begin{equation}
    \sum_{K\in\mathcal{T}}\int_{\partial K} \del{-\nabla u + \bar{p}I} n
    \cdot\bar{v}_h\dif s = \int_{\Gamma}\del{-\nabla u + \bar{p}I} n
    \cdot\bar{v}_h\dif s = 0,
  \end{equation}
  where the first equality is due to the single-valuedness of $u$,
  $\bar{p}$ and $\bar{v}_h$ on element boundaries, and the second
  equality is due to $\bar{v}_h = 0$ on~$\Gamma$. We therefore
  conclude for the momentum equation that
  \begin{equation}
    a_h({\bf u}, {\bf v}_h) + b_h({\bf p}, {\bf v}_h)
      = \int_{\Omega} f \cdot v_{h} \dif x.
  \end{equation}

  We next consider the continuity equation. First note that
  \begin{equation}
    b_h({\bf q}_h, {\bf u}) =
    - \int_{\Omega}q_h \nabla \cdot u\dif x
    + \sum_{K\in\mathcal{T}}\int_{\partial K}(u-\bar{u})\cdot n
    \bar{q}\dif s = 0,
  \end{equation}
  because $\bar{u} = \gamma(u)$ and~$\nabla \cdot u = 0$. Furthermore,
  \begin{equation}
    c_h({\bf p}, {\bf q}_h) =
    \sum_{K\in\mathcal{T}} \int_{\partial K} \alpha_{p} h_K (\bar{p}-p) (\bar{q}_h-q_h)\dif s = 0,
  \end{equation}
  because~$\bar{p} = \gamma(p)$. It follows that
  \begin{equation}
    - b_h({\bf q}_h, {\bf u}) + c_h({\bf p}, {\bf q}_h) = 0,
  \end{equation}
  concluding the proof.
\end{proof}

\subsection{Stability and boundedness of the vector-Laplacian term}

Some results from \cite{Wells:2011} are generalised in this section
to the vector-Laplacian term $a_h$, and are provided here for
completeness.

\begin{lemma}[Stability of $a_h$]
  \label{lem:stab_ah}
  There exists a $\beta_{v} > 0$, independent of $h$, and a constant
  $\alpha_0 > 0$ such that for $\alpha_{v} > \alpha_0$ and for
  all~${\bf v}_h\in V_h^{\star}$
  \begin{equation}
    \label{eq:stab_ah}
    a_h({\bf v}_h, {\bf v}_h) \ge \beta_v \tnorm{{\bf v}_h}^2_v.
  \end{equation}
\end{lemma}

\begin{proof}
  By definition of the bilinear form $a_h$ in~\cref{eq:formA},
  \begin{multline}
    \label{eq:ahvvdef}
    a_h({\bf v}_h, {\bf v}_h)
    =
    \sum_{K\in\mathcal{T}}\norm{\nabla v_h}_{0,K}^2
    + \sum_{K\in\mathcal{T}}\frac{\alpha_{v}}{h_K}\norm{v_h - \bar{v}_h}^2_{0,\partial K}
    \\
    + 2\sum_{K\in\mathcal{T}}\int_{\partial K}(\bar{v}_h-v_h)\cdot \frac{\partial v_h}{\partial n} \dif s.
  \end{multline}
  Applying the Cauchy--Schwarz inequality and a trace inequality to
  the third term on the right-hand side of~\cref{eq:ahvvdef},
  \begin{equation}
    \begin{split}
      \envert{2\sum_{K\in\mathcal{T}}\int_{\partial K}(\bar{v}_h-v_h)\cdot \frac{\partial v_h}{\partial n} \dif s}
      &\le 2\frac{h_K^{1/2}}{\alpha_v^{1/2}}\norm{\frac{\partial v_h}{\partial n}}_{0,\partial K}
      \frac{\alpha_v^{1/2}}{h_K^{1/2}}\norm{\bar{v}_h-v_h}_{0,\partial K}
      \\
      &\le 2 c\alpha_v^{-1/2} \norm{\nabla v_h}_{0,K}\frac{\alpha_v^{1/2}}{h_K^{1/2}}\norm{\bar{v}_h-v_h}_{0,\partial K}.
    \end{split}
  \end{equation}
  Combined with~\cref{eq:ahvvdef} we obtain
  \begin{equation}
    \label{eq:ahvvdefineq}
    a_h({\bf v}_h, {\bf v}_h)
    \ge
    \sum_{K\in\mathcal{T}} \del{
      \norm{\nabla v_h}_{0,K}^2
      + 2 c \alpha_v^{-1/2} \norm{\nabla v_h}_{0,K}\frac{\alpha_v^{1/2}}{h_K^{1/2}}\norm{\bar{v}_h-v_h}_{0,\partial K}
      + \frac{\alpha_{v}}{h_K}\norm{v_h - \bar{v}_h}^2_{0,\partial K}}.
  \end{equation}
  Note that for any $0 < \Psi < 1$ the following inequality holds for
  $x,y\in \mathbb{R}$: $x^2 - 2 \Psi xy + y^2 \ge
  \tfrac{1}{2}(1-\Psi^2)(x^2 + y^2)$~\cite{Pietro:book}.  Taking $x =
  \norm{\nabla v_h}_{0, K}$, $y = \alpha_{v}^{1/2} h_K^{-1/2}
  \norm{v_h - \bar{v}_h}_{0, \partial K}$ and $\Psi = c
  \alpha_v^{-1/2}$, then if $\alpha_v > c^2 = \alpha_0$ it follows
  that
  \begin{equation}
    \label{eq:ahvvdefineqend}
    a_h({\bf v}_h, {\bf v}_h) \ge
    \tfrac{1}{2}(1-\alpha_0/\alpha_v)\tnorm{ {\bf v}_h }^2_v,
  \end{equation}
  so that the result follows with
  $\beta_v = \tfrac{1}{2}(1-\alpha_0/\alpha_v)$.
\end{proof}

\begin{lemma}[Boundedness of $a_h$]
  \label{lem:bound_ah}
  There exists a $c > 0$, independent of $h$, such that for all ${\bf
    u}\in V^{\star}(h)$ and for all ${\bf v}_h\in V_h^{\star}$
  \begin{equation}
    \label{eq:bound_ah}
    \envert{a_h({\bf u}, {\bf v}_h)}
    \le C_a\tnorm{{\bf u}}_{v'}\tnorm{{\bf v}_h}_v,
  \end{equation}
  with $C_a = c(1+\alpha_v^{-1/2})$.
\end{lemma}

\begin{proof}
  From the definition of $a_h$ in~\cref{eq:formA},
  \begin{multline}
    a_h({\bf u}, {\bf v}_h) =
    \underbrace{\sum_{K} \int_{K}\nabla u : \nabla v_h \dif x}_{T_{1}}
    - \underbrace{\sum_{K\in\mathcal{T}}\int_{\partial K} \frac{\partial u}{\partial n}\cdot(v_h-\bar{v}_h) \dif s}_{T_{2}}
\\
- \underbrace{\sum_{K\in\mathcal{T}}\int_{\partial K}(u - \bar{u})
  \cdot \frac{\partial v_h}{\partial n} \dif s}_{T_{3}}
    + \underbrace{\sum_{K\in\mathcal{T}}\int_{\partial K}\frac{\alpha_{v}}{h_K}(u - \bar{u})\cdot
    (v_h - \bar{v}_h) \dif s}_{T_{4}}.
  \end{multline}
  A bound for $T_3$ follows from
  \begin{equation}
    \begin{split}
      \envert{T_3} &\le \del{\sum_{K\in\mathcal{T}}\frac{h_K}{\alpha_v}\norm{\frac{\partial v_h}{\partial n}}^2_{0,\partial K}}^{1/2}
      \del{\sum_{K\in\mathcal{T}}\frac{\alpha_v}{h_K}\norm{\bar{u}-u}^2_{0,\partial K}}^{1/2}
      \\
      &\le c\alpha_v^{-1/2}\del{\sum_{K\in\mathcal{T}}\norm{\nabla v_h}^2_{0, K}}^{1/2}
      \tnorm{{\bf u}}_v
      \\
      &\le c\alpha_v^{-1/2}\tnorm{{\bf v}_h}_v\tnorm{{\bf u}}_{v'},
    \end{split}
    \label{eqn:t3-bound}
  \end{equation}
  where $c > 0$ is a constant independent of~$h$. For the second
  inequality in~\cref{eqn:t3-bound} we used the discrete trace
  inequality~\cref{eq:trace_inequality}.  Similar bounds for $T_1$,
  $T_2$ and $T_4$ follow after applying the Cauchy--Schwarz
  inequality. Collecting all bounds proves~\cref{eq:bound_ah}.
\end{proof}

\subsection{Stability of the pressure--velocity coupling term}

We now examine stability of the discrete pressure--velocity coupling
term~$b_h$. The analysis of~$b_h$ for the equal-order and mixed-order
cases differs, hence we will prove stability of~$b_h$ for the two
cases separately.

To prove stability of the discrete pressure--velocity coupling term,
$b_h$, we remark that satisfaction of the inf-sup condition for the
infinite-dimensional problem~\cref{eq:infsupcondition_continuous} is
equivalent to there existing for all $q \in L^{2}_{0}(\Omega)$ a
$v_q \in \sbr{H^1_0(\Omega)}^d$ that satisfies
\begin{equation}
  \label{eq:stab_b}
  q = \nabla\cdot v_q \quad \mbox{and} \quad
  \beta_c\norm{v_q}_{1,\Omega} \le \norm{q}_{0,\Omega}
\end{equation}
(see, e.g.~\cite[Theorem~6.5]{Pietro:book}). We make extensive use of
this result.

We state now the stability lemma for the pressure--velocity coupling.
\begin{lemma}[Stability of $b_h$]
  \label{lem:stab_bh}
  There exists a constant $\beta_p > 0$, independent of $h$, such that
  for all~${\bf q}_h \in Q_h^{\star}$
  \begin{equation}
    \label{eq:stab_bh}
    \beta_p \norm{q_h}_{0,\Omega} \le \sup_{{\bf w}_h\in V_h^{\star}}
    \frac{ b_h({\bf q}_h, {\bf w}_h) }{\tnorm{ {\bf w}_h }_v}
    + \envert{{\bf q}_h}_p.
  \end{equation}
\end{lemma}
We prove the above lemma for the equal-order and mixed-order cases in
the following sections. Recall that $\envert{{\bf q}_h}_p$ is zero for
the mixed-order case since~$\alpha_{p} = 0$.

\subsubsection{Equal-order case}
For the equal-order case ($k = m$ in~\cref{eqn:spaces_cell}), we
introduce the projections $\Pi_h$ and $\bar{\Pi}_h$, where
$\Pi_h : [H^1(\Omega)]^d \to V_h$ is any projection such that
\begin{equation}
  \label{eq:Piprojection_K}
 \int_{K}(\Pi_h v - v)\cdot y_h \dif x = 0 \qquad
  \forall y_h \in [P_{k - 1}(K)]^{d}
\end{equation}
for all $K \in \mathcal{T}_h$, and
$\bar{\Pi}_h : [H^1(\Omega)]^d \to \Bar{V}_h$ is the $L^2$-projection
into~$\bar{V}_h$,
\begin{equation}
  \sum_{K\in\mathcal{T}}\int_{\partial K}(\bar{\Pi}_hv - v) \cdot \bar{y}_h \dif s
  = 0 \qquad \forall \bar{y}_h\in \bar{V}_h.
\end{equation}
The following two inequalities will be used below to prove stability
of~$b_h$:
\begin{align}
  \label{eq:Pi_face_boundaryK}
  \norm{w - \Pi_hw}_{0,F} \le& ch_K^{1/2} \envert{w}_{1,K} &&
  \forall F\in\mathcal{F},\ F\subset\partial K,\ \forall K\in\mathcal{T}_h,
  \\
  \label{eq:Pi_barPi_inequality}
  \norm{\Pi_hw - \bar{\Pi}_hw}_{\partial K}^2 \le& h_K\norm{w}^2_{1,K} &&
  \forall K\in\mathcal{T}_h,
\end{align}
where $c > 0$ is independent of~$h$. The first inequality is due to
\cite[Lemma~1.59]{Pietro:book}, and the second is due
to~\cite[Proposition~3.9]{Cockburn:2011}.

We can now prove stability of the discrete pressure--velocity coupling
term, $b_h$, for an equal-order velocity/pressure approximation.
\begin{proof}[Proof of \cref{lem:stab_bh} for the equal-order case]
  For a $q_h \in Q_h$, from~\cref{eq:stab_b}, there exists a $v_{q_h}
  \in \sbr{H_0^1(\Omega)}^d$ such that $\nabla \cdot v_{q_h} = q_h$
  and $\beta_c \norm{v_{q_h}}_{1, \Omega} \le \norm{q_h}_{0,
    \Omega}$. Since $v_{q_h}\in \sbr{H^1_0(\Omega)}^d$ we see that
  \begin{multline}
    \norm{q_h}^2_{0,\Omega} = \int_{\Omega} q_h^2 \dif x
    = \int_{\Omega}q_h\nabla\cdot v_{q_h}\dif x
\\
    = -\sum_{K\in\mathcal{T}}\int_K\nabla q_h\cdot v_{q_h}\dif x
    + \sum_{K\in\mathcal{T}}\int_{\partial K}q_hv_{q_h}\cdot n \dif s.
  \end{multline}
  By~\cref{eq:Piprojection_K}, we note that $\sum_{K\in\mathcal{T}}
  \int_{K} \nabla q_h \cdot (\Pi_hv_{q_h} - v_{q_h}) \dif x = 0$, and
  it follows that
  \begin{equation}
    \label{eq:intbypartsnormqh_Pi}
    \norm{q_h}^2_{0,\Omega}
    = -\sum_{K\in\mathcal{T}}\int_K\nabla q_h\cdot \Pi_hv_{q_h}\dif x
    + \sum_{K\in\mathcal{T}}\int_{\partial K}q_hv_{q_h}\cdot n\dif s.
  \end{equation}

  Next, we note for ${\bf q}_{h} = (q_{h}, \bar{q}_{h}) \in
  Q^{*}_{h}$, from the definition of $b_{h}$ in~\cref{eq:formB} and
  applying integration by parts we have
  \begin{equation}
    \begin{split}
    b_h({\bf q}_h, (\Pi_hv_{q_h},\bar{\Pi}_hv_{q_h})) = &
    \int_{\Omega}\nabla q_h  \cdot \Pi_hv_{q_h}\dif x
    + \sum_{K\in\mathcal{T}}\int_{\partial K}\sbr{(\Pi_hv_{q_h} - \bar{\Pi}_hv_{q_h}) \cdot n
      \bar{q}_h - \Pi_hv_{q_h}\cdot n q_h} \dif s
    \\
    = &
    \int_{\Omega}\nabla q_h  \cdot \Pi_hv_{q_h}\dif x
    + \sum_{K\in\mathcal{T}}\int_{\partial K}(\bar{q}_h-q_h)\Pi_hv_{q_h}\cdot n \dif s,
    \end{split}
  \end{equation}
  where the last equality is due to the single-valuedness of
  $\bar{\Pi}_h v_{q_h}$ and $\bar{q}_h$ across element boundaries and
  because $\bar{\Pi}_h v_{q_h} = 0$ on the domain
  boundary~$\Gamma$. We may now write~\cref{eq:intbypartsnormqh_Pi}
  as
  \begin{equation}
    \begin{split}
    \norm{q_h}^2_{0, \Omega}
    =& -b_h({\bf q}_h, (\Pi_hv_{q_h},\bar{\Pi}_hv_{q_h}))
    + \sum_{K \in\mathcal{T}}\int_{\partial K}(\bar{q}_h-q_h)\Pi_h v_{q_h} \cdot n\dif s
    \\
    &- \sum_{K\in\mathcal{T}}\int_{\partial K}(\bar{q}_h-q_h)v_{q_h}\cdot n\dif s,
    \end{split}
  \end{equation}
  where equality is due to the single-valuedness of $v_{q_h}$ and
  $\bar{q}_h$ across element boundaries and because $v_{q_h} = 0$ on
  the domain boundary~$\Gamma$. It follows then that
  \begin{equation}
    \label{eq:qhnormT1T2}
    \norm{q_h}^2_{0,\Omega}
    =
    \underbrace{-b_h({\bf q}_h, (\Pi_hv_{q_h},\bar{\Pi}_hv_{q_h}))}_{T_{1}}
    +
    \underbrace{\sum_{K\in\mathcal{T}}\int_{\partial K}(\bar{q}_h-q_h)(\Pi_hv_{q_h}-v_{q_h})\cdot n\dif s}_{T_{2}}.
  \end{equation}
  We now bound $T_1$ and $T_2$ separately. Starting with~$T_1$,
  \begin{equation}
    \envert{T_1}
    \le \del{
      \sup_{{\bf w}_h\in V_h^{\star}}
      \frac{b_h({\bf q}_h, {\bf w}_h)}
      {\tnorm{{\bf w}_h}_v}} \tnorm{(\Pi_hv_{q_h},\bar{\Pi}_hv_{q_h})}_v.
  \end{equation}
  Noting that
  \begin{equation}
    \tnorm{(\Pi_hv_{q_h},\bar{\Pi}_hv_{q_h})}_v^2
    = \sum_{K\in\mathcal{T}}\norm{\nabla(\Pi_hv_{q_h})}^2_{0,K} +
    \sum_{K\in\mathcal{T}}\frac{\alpha_{v}}{h_K}\norm{\bar{\Pi}_hv_{q_h} - \Pi_hv_{q_h}}^2_{0,\partial K},
  \end{equation}
  it is possible to bound the two terms on the right separately. By
  \cite[Lemma~6.11]{Pietro:book} and~\cref{eq:Pi_barPi_inequality}
  we find, respectively,
  \begin{equation}
    \begin{split}
      \sum_{K\in\mathcal{T}}\norm{\nabla(\Pi_hv_{q_h})}^2_{0,K} \le& c \norm{v_{q_h}}^2_{1,\Omega},
      \\
      \sum_{K\in\mathcal{T}}\frac{\alpha_v}{h_K}\norm{\bar{\Pi}_hv_{q_h} - \Pi_hv_{q_h}}^2_{0,\partial K}
      \le& c\alpha_v\norm{v_{q_h}}^2_{1,\Omega},
    \end{split}
  \end{equation}
  with $c > 0$ independent of~$h$. Using~\cref{eq:stab_b} it follows
  then that
  \begin{equation}
    \label{eq:bound_for_T1}
    \envert{T_1}
    \le c(1 + \alpha_v) \beta_c^{-1}
    \del{
      \sup_{{\bf w}_h\in V_h^{\star}}
      \frac{b_h({\bf q}_h, {\bf w}_h)}
      {\tnorm{{\bf w}_h}_v}} \norm{q_h}_{0,\Omega}.
  \end{equation}
  We now bound~$T_2$. Using the Cauchy--Schwarz inequality,
  \begin{equation}
    \begin{split}
      \envert{T_2}
      &\le
      \del{\sum_{K\in\mathcal{T}}\int_{\partial K}h_K(\bar{q}_h-q_h)^2 \dif s}^{1/2}
      \del{\sum_{K\in\mathcal{T}} \int_{\partial K} h_K^{-1}
      \envert{\Pi_hv_{q_h}-v_{q_h}}^2\dif s}^{1/2}
      \\
      &\le \alpha_p^{-1} \envert{{\bf q}_h}_p
      \del{\sum_{K\in\mathcal{T}}h_K^{-1}\norm{\Pi_hv_{q_h} - v_{q_h}}^2_{0, \partial K}}^{1/2},
    \end{split}
  \end{equation}
  where the last inequality follows from the definition of the
  pressure semi-norm in~\cref{eq:pressureseminorm}. Note that
  by~\cref{eq:Pi_face_boundaryK}
  \begin{equation}
    \norm{\Pi_hv_{q_h}-v_{q_h}}^2_{0,\partial K}
    \le \sum_{F\in\mathcal{F},\ F\subset\partial K}\norm{\Pi_hv_{q_h}-v_{q_h}}^2_{0,F}
    \le c h_K \envert{v_{q_h}}^2_{1,K},
  \end{equation}
  where $c > 0$ is independent of~$h$. Using again~\cref{eq:stab_b}
  we find the following bound for~$T_2$,
  \begin{equation}
    \label{eq:bound_for_T2}
    \envert{T_2}
    \le  c\alpha_p^{-1} \envert{{\bf q}_h}_p\norm{v_{q_h}}_{1,\Omega}
    \le c(\alpha_p\beta_c)^{-1} \envert{{\bf q}_h}_p\norm{q_h}_{0,\Omega}.
  \end{equation}

  Combining now~\cref{eq:qhnormT1T2}, \cref{eq:bound_for_T1}
  and~\cref{eq:bound_for_T2}, we find
  \begin{equation}
    \beta_c\norm{q_h}^2_{0,\Omega} \le
    c(1+\alpha_v)\del{
      \sup_{{\bf w}_h\in V_h^{\star}}
      \frac{b_h({\bf q}_h, {\bf w}_h)}
      {\tnorm{{\bf w}_h}_v}} \norm{q_h}_{0,\Omega}
    + c\alpha_p^{-1} \envert{{\bf q}_h}_p\norm{q_h}_{0,\Omega}.
  \end{equation}
  Dividing both sides by $\norm{q_h}_{0,\Omega}$ and rearranging terms
  the bound in~\cref{eq:stab_bh} follows with
  \begin{equation}
    \beta_p = \frac{c\beta_c}{\max(1+\alpha_v,\alpha_p^{-1})}.
  \end{equation}
\end{proof}
Note that $\beta_{p} \rightarrow 0$ as $\alpha_{p} \rightarrow 0$,
illustrating the need for the penalty term on the pressure jump.

\subsubsection{Mixed-order case}
\label{ss:mixedorderinfsup}

To prove \cref{lem:stab_bh} for a mixed-order velocity--pressure
approximation, we follow a similar approach to \cite{Hansbo:2002}. For
this we require the definition of the BDM interpolation operator, as
given in the following lemma~\cite[Lemma~7]{Hansbo:2002}. See
\cite{schotzau:2002} for the analogous operators on quadrilaterals and
hexahedra.

\begin{lemma}
  \label{lem:BDMprojection}
  If the mesh consists of triangles in two dimensions or tetrahedra in
  three dimensions there is an interpolation operator $\Pi_{\rm BDM} :
  [H^1(\Omega)]^d \to V_h$ with the following properties for all $u\in
  [H^{k+1}(K)]^d$:
  \begin{itemize}
  \item[(i)] $\jump{n\cdot\Pi_{\rm BDM} u} = 0$, where $\jump{a} =
    a^{+} +a^{-}$ and $\jump{a}=a$ on, respectively, interior and
    boundary faces is the usual jump operator.
  \item[(ii)] $\norm{u-\Pi_{\rm BDM} u}_{m,K} \le ch_K^{l - m}
    \norm{u}_{l,K}$ with $m = 0, 1, 2$ and $m\le l \le k+1$.
  \item[(iii)] $\norm{\nabla\cdot(u - \Pi_{\rm BDM} u)}_{m, K} \le c
    h_K^{l - m} \norm{\nabla\cdot u}_{l, K}$ with $m = 0, 1$ and $m
    \le l \le k$.
  \item[(iv)] $\int_K q(\nabla\cdot u - \nabla\cdot\Pi_{\rm BDM}u)
    \dif x = 0$ for all $q \in P_{k - 1}(K)$.
  \item[(v)] $\int_F \bar{q}(n\cdot u - n\cdot \Pi_{\rm BDM} u) \dif s
    = 0$ for all $\bar{q}\in P_k(F)$, where $F$ is a face on~$\partial
    K$.
  \end{itemize}
\end{lemma}

We can now prove stability of the discrete pressure--velocity coupling
term, $b_h$, for a mixed-order velocity/pressure approximation.
\begin{proof}[Proof of \cref{lem:stab_bh} for the mixed-order case]
  Let $q_h \in Q_h$. By~\cref{eq:stab_b} there exists a
  $v_{q_h}\in \sbr{H_0^1(\Omega)}^d$ such that $\nabla \cdot v_{q_h} =
  q_h$ and $\beta_c \norm{v_{q_h}}_{1, \Omega} \le \norm{q_h}_{0,
    \Omega}$. Since $v_{q_h} \in \sbr{H^1_0(\Omega)}^d$ we see that
  \begin{equation}
    \label{eq:normqhinBDMbh}
    \norm{q_h}^2_{0,\Omega}
    = \int_{\Omega}q_h\nabla\cdot v_{q_h}\dif x
    = \int_{\Omega}q_h\nabla\cdot \Pi_{\rm BDM}v_{q_h}\dif x
    = -b_h({\bf q}_h, (\Pi_{\rm BDM}v_{q_h}, \bar{\Pi}_hv_{q_h})),
  \end{equation}
  by $(iv)$ of \cref{lem:BDMprojection} and by definition of $b_h$
  in~\cref{eq:formB}. Next, we determine a bound for
  $\tnorm{(\Pi_{\rm BDM}v_{q_h}, \bar{\Pi}_h v_{q_h})}_v$. We first
  note that
  \begin{equation}
    \label{eq:elementbdmbound}
    \norm{\nabla(\Pi_{\rm BDM}v_{q_h})}_{0, K}
    \le
    \norm{\nabla v_{q_h} - \nabla(\Pi_{\rm BDM} v_{q_h})}_{0, K} +
    \norm{\nabla v_{q_h}}_{0,K}
    \le
    c \norm{v_{q_h}}_{1, K},
  \end{equation}
  due to $(ii)$ of \cref{lem:BDMprojection}. We also note that
  \begin{multline}
    \label{eq:boundingbdmboundary}
    h_K^{-1}\norm{\Pi_{\rm BDM}v_{q_h}-\bar{\Pi}_hv_{q_h}}^2_{0,\partial K}
    \le
    h_K^{-1}\norm{\Pi_{\rm BDM}v_{q_h}-v_{q_h}}^2_{0,\partial K}
    \\
    + h_K^{-1}\norm{\bar{\Pi}_hv_{q_h} - \Pi_hv_{q_h}}^2_{0,\partial K} +
    h_K^{-1}\norm{\Pi_hv_{q_h}-v_{q_h}}^2_{0,\partial K}.
  \end{multline}
  By the trace inequality~\cref{eq:trace_inequality}, item $(ii)$ of
  \cref{lem:BDMprojection}, and the
  inequalities~\cref{eq:Pi_face_boundaryK}
  and~\cref{eq:Pi_barPi_inequality},
  \begin{equation}
    \begin{split}
      h_K^{-1}\norm{\Pi_{\rm BDM}v_{q_h}-v_{q_h}}^2_{0,\partial K}
      &\le
      c h_K^{-2}\norm{\Pi_{\rm BDM}v_{q_h}-v_{q_h}}^2_{0, K}
      \le c\norm{v_{q_h}}^2_{1,K},
      \\
      h_K^{-1}\norm{\Pi_hv_{q_h}-v_{q_h}}^2_{0,\partial K}
      &\le c \envert{v_{q_h}}^2_{1,K},
      \\
      h_K^{-1}\norm{\bar{\Pi}_hv_{q_h} - \Pi_hv_{q_h}}^2_{0,\partial K}
        &\le \norm{v_{q_h}}^2_{1,K}.
    \end{split}
  \end{equation}
  Combining these results with~\cref{eq:boundingbdmboundary}, we obtain
  \begin{equation}
    \label{eq:boundingbdmboundary_final}
    \frac{\alpha_v}{h_K}\norm{\Pi_{\rm  BDM}v_{q_h}-\bar{\Pi}_hv_{q_h}}^2_{0,\partial K} \le
    c \alpha_v \norm{v_{q_h}}^2_{1,K}.
  \end{equation}
  From~\cref{eq:elementbdmbound}
  and~\cref{eq:boundingbdmboundary_final} we therefore find that
  \begin{equation}
    \label{eq:tnormbdmfinalbound}
    \tnorm{(\Pi_{\rm BDM}v_{q_h}, \bar{\Pi}_hv_{q_h})}_v^2 \le c(1+\alpha_v)\norm{v_{q_h}}^2_{1,K}.
  \end{equation}
  Satisfaction of the inf-sup condition follows from
  \begin{equation}
    \label{eq:fininfsupql2}
    \sup_{{\bf w}_h\in V_h^{\star}}
    \frac{ -b_h({\bf q}_h, {\bf w}_h) }{\tnorm{ {\bf w}_h }_v}
    \ge
    \frac{ -b_h({\bf q}_h, (\Pi_{\rm BDM}v_{q_h}, \bar{\Pi}_hv_{q_h})) }{\tnorm{ (\Pi_{\rm BDM}v_{q_h}, \bar{\Pi}_hv_{q_h}) }_v}
    \ge
    \frac{c\beta_c}{1+\alpha_v}\norm{q_h}_{0,\Omega},
  \end{equation}
  where we have
  used~\cref{eq:normqhinBDMbh},~\cref{eq:tnormbdmfinalbound}
  and~\cref{eq:stab_b} for the second inequality. The bound
  in~\cref{eq:stab_bh} follows with
  $\beta_p = c\beta_c/(1+\alpha_v)$.
\end{proof}
Note that the analysis for the mixed order case does not depend on the
pressure penalty term~$\alpha_{p}$, which can be set to zero.

By \cref{lem:stab_ah,lem:stab_bh}, it straightforward to show that a
solution $({\bf u}_h, p_h)$ to the $H(\rm div)$-type formulation in
\cref{eq:bdm_compact_wf} is unique. However, for the case $\alpha_{p}
= 0$, \cref{lem:stab_bh} does not involve any norms of $\Bar{q}_{h}
\in \Bar{Q}_{h}$. The following proposition shows that
\cref{lem:stab_bh} is sufficient for the formulation in
\cref{eq:compact_wf}.

\begin{proposition}
  \label{prop:bdm_hyb_solution}
  If $\alpha_v > \alpha_0$, a solution $({\bf u}_h, {\bf p}_h) \in
  X_{h}^{\star}$ to \cref{eq:compact_wf} is unique.
  \end{proposition}
\begin{proof}
  We wish to show that for zero data ${\bf u}_h = \boldsymbol{0}$ and
  ${\bf p}_h = \boldsymbol{0}$.  Setting ${\bf v}_h = {\bf u}_h$ and
  ${\bf q}_h = {\bf p}_h$ in \cref{eq:compact_wf}, coercivity of
  $a_h(\cdot, \cdot)$ implies that ${\bf u}_h = \boldsymbol{0}$.

  Substituting ${\bf u}_h = \boldsymbol{0}$ into \cref{eq:compact_wf}
  results in
  \begin{equation}
    \label{eq:bhdomain}
    -\sum_{K\in\mathcal{T}}\int_K p_h\nabla\cdot v_h \dif x +
    \sum_{K\in\mathcal{T}}\int_{\partial K}v_h\cdot n \bar{p}_h \dif s = 0 \qquad \forall v_h \in V_{h}.
  \end{equation}
  Setting $v_h = 0$ on all elements except $K$, and integrating by
  parts,
  \begin{equation}
    \label{eq:bhK}
    \int_K \nabla p_h \cdot v_h \dif x +
    \int_{\partial K}v_h\cdot n \del{\bar{p}_h - p_h} \dif s = 0
    \qquad \forall v_h \in \sbr{P_k(K)}^d.
  \end{equation}
  We consider a $w_h \in \sbr{P_k(K)}^3$ that satisfies
  \begin{subequations}
    \begin{align}
       \label{eq:wh_partialK}
      \int_{\partial K} w_h\cdot n \bar{r}_h \dif s
      &= \int_{\partial K} \del{p_h - \bar{p}_h} \bar{r}_h \dif s
      \quad&& \forall \bar{r}_h \in P_k(\partial K),
      \\
      \label{eq:wh_K}
      \int_K w_h \cdot z_h \dif x
      &= 0 \quad&& \forall z_h \in \mathcal{N}_{k-2}(K),
     \end{align}
    \label{eq:constructwh}
  \end{subequations}
  where $\mathcal{N}_{k-2}$ is the N\'ed\'elec
  space~\cite{Boffi:book}. We remark that such a $w_h$ exists and is
  unique~\cite[Proposition 2.3.2]{Boffi:book}. Using $w_h$ as test
  function in \cref{eq:bhK},
  \begin{equation}
    \label{eq:bhK_wh}
    0 = \int_K \nabla p_h \cdot w_h \dif x +
    \int_{\partial K}w_h\cdot n \del{\bar{p}_h - p_h} \dif s
    = \int_{\partial K} \del{\bar{p}_h - p_h}^2 \dif s,
  \end{equation}
  where we used that $\nabla p_h \in \nabla P_{k-1}(K) \subset
  \sbr{P_{k-2}(K)}^3 \subset \mathcal{N}_{k-2}(K)$, and set $\bar{r}_h
  = \bar{p}_h - p_h \in P_k(\partial K)$. \Cref{eq:bhK_wh} implies
  that $p_h = \bar{p}_h$ on $\partial K$.  Using $p_{h} = \Bar{p}_{h}$
  on facets in \cref{eq:bhK} shows that the pressure $p_{h}$ is
  defined only up to a constant.  Since $p_{h} = \Bar{p}_{h}$, and
  considering that $\Bar{p}_{h}$ is single-valued on facets, together
  with $\int_{\Omega}p_h \dif x = 0$ in \cref{eq:stokes_d} we obtain
  that~${\bf p}_h = \boldsymbol{0}$.
\end{proof}

We have proved existence and uniqueness of a solution to
\cref{eq:compact_wf} for simplex elements (the degrees-of-freedom in
\cref{eq:constructwh} are specific to simplices). For other elements,
the analogous degree-of-freedom definitions can be found in
\cite[Chapter 2]{Boffi:book}.

\subsection{Well-posedness and boundedness}

We are now ready to prove inf-sup stability and boundedness of the
method. The analysis for the equal- and mixed-order case is identical,
building on \cref{lem:stab_bh}. We first prove satisfaction of
the discrete inf-sup condition.
\begin{lemma}[Discrete inf-sup stability]
  \label{lem:stab_isfem}
  If $\alpha_{v} > \alpha_0$, with $\alpha_0$ defined in
  \cref{lem:stab_ah}, then there exists a constant $\sigma > 0$,
  independent of $h$, such that for all~$({\bf v}_h,{\bf q}_h)\in
  X_h^{\star}$
  \begin{equation}
    \sigma\tnorm{({\bf v}_h,{\bf q}_h)}_{v,p}
    \le
    \sup_{({\bf w}_h,{\bf r}_h)\in X_h^{\star}}
    \frac{B_h(({\bf v}_h,{\bf q}_h),({\bf w}_h,{\bf r}_h))}
    {\tnorm{({\bf w}_h,{\bf r}_h)}_{v,p}}.
  \end{equation}
\end{lemma}
\begin{proof}
  We first note that
  \begin{equation}
    B_h(({\bf v}_h,{\bf q}_h),({\bf v}_h,{\bf q}_h))
    = a_h({\bf v}_h, {\bf v}_h)
    + c_h({\bf q}_h, {\bf q}_h)
    \ge \beta_v\tnorm{{\bf v}_h}^2_v + \envert{{\bf q}_h}_p^2,
  \end{equation}
  by \cref{lem:stab_ah} and by definition of
  $c_h$~\cref{eq:formC} and the pressure
  semi-norm~\cref{eq:pressureseminorm}. Then,
  \begin{equation}
    \label{eq:betavvhCqh}
    \beta_v\tnorm{{\bf v}_h}^2_v + \envert{{\bf q}_h}_p^2
    \le
    \sup_{({\bf w}_h,{\bf r}_h)\in X_h^{\star}}
    \frac{B_h(({\bf v}_h,{\bf q}_h),({\bf w}_h,{\bf r}_h))}
    {\tnorm{({\bf w}_h,{\bf r}_h)}_{v,p}}
    \tnorm{({\bf v}_h,{\bf q}_h)}_{v,p}.
  \end{equation}
  It is clear that $b_h({\bf q}_h, {\bf w}_h) = B_h(({\bf v}_h, {\bf
    q}_h),({\bf w}_h,{\bf 0})) - a_h({\bf v}_h, {\bf w}_h)$, so using
  \cref{lem:stab_bh} we find that
  \begin{equation}
    \begin{split}
      \beta_p \norm{q_h}_{0,\Omega}
      &\le
      \sup_{{\bf w}_h\in V_h^{\star}} \del{
        \frac{ -a_h({\bf v}_h, {\bf w}_h) }{\tnorm{ {\bf w}_h }_v} +
    \frac{ B_h(({\bf v}_h, {\bf q}_h),({\bf w}_h,{\bf 0})) }{\tnorm{ ({\bf w}_h,{\bf 0}) }_{v,p}}}
      + \envert{{\bf q}_h}_p
      \\
      &\le
      \sup_{{\bf w}_h\in V_h^{\star}}
        \frac{ a_h({\bf v}_h, {\bf w}_h) }{\tnorm{ {\bf w}_h }_v}
        +
        \sup_{({\bf w}_h,{\bf r}_h)\in X_h^{\star}}
        \frac{B_h(({\bf v}_h,{\bf q}_h),({\bf w}_h,{\bf r}_h))}
        {\tnorm{({\bf w}_h,{\bf r}_h)}_{v,p}}
      + \envert{{\bf q}_h}_p.
    \end{split}
  \end{equation}
  Using the boundedness of~$a_h$ (\cref{lem:bound_ah})
  and~\cref{eq:equivalentNorms_v_vprime},
  \begin{equation}
    \sup_{{\bf w}_h\in V_h^{\star}}
    \frac{ a_h({\bf v}_h, {\bf w}_h) }{\tnorm{ {\bf w}_h }_v}
    \le cC_a(1+\alpha_v^{-1})\tnorm{ {\bf v}_h }_{v},
  \end{equation}
  where $c > 0$ is independent of~$h$. Let
  $c_{\alpha_v}=cC_a(1+\alpha_v^{-1})$. It follows then that
  \begin{equation}
    \label{eq:betapqhle}
    \beta_p \norm{q_h}_{0,\Omega} \le
      c_{\alpha_v}\tnorm{ {\bf v}_h }_{v}
      +
      \sup_{({\bf w}_h,{\bf r}_h)\in X_h^{\star}}
      \frac{B_h(({\bf v}_h,{\bf q}_h),({\bf w}_h,{\bf r}_h))}
      {\tnorm{({\bf w}_h,{\bf r}_h)}_{v,p}}
      + \envert{{\bf q}_h}_p.
  \end{equation}
  Applying Young's inequality twice, it follows that
  \begin{multline}
    \label{eq:betapqhleYoung}
    \beta_p^{2}\norm{q_h}^2_{0,\Omega} \le
    4\del{ c_{\alpha_v}^2\tnorm{ {\bf v}_h }_{v}^2 + \envert{{\bf q}_h}^2_p}
    \\
    + 2 \del{\sup_{({\bf w}_h,{\bf r}_h)\in X_h^{\star}}
    \frac{B_h(({\bf v}_h,{\bf q}_h),({\bf w}_h,{\bf r}_h))}
         {\tnorm{({\bf w}_h,{\bf r}_h)}_{v,p}}}^2.
  \end{multline}
  Applying Young's inequality now to~\cref{eq:betavvhCqh},
  \begin{multline}
    \label{eq:betavvhCqhY}
    \beta_v\tnorm{{\bf v}_h}^2_v + \envert{{\bf q}_h}_p^2
    \\
    \le
    \frac{1}{2\psi}\del{
      \sup_{({\bf w}_h,{\bf r}_h)\in X_h^{\star}}
      \frac{B_h(({\bf v}_h,{\bf q}_h),({\bf w}_h,{\bf r}_h))}
      {\tnorm{({\bf w}_h,{\bf r}_h)}_{v,p}}}^2
    + \frac{\psi}{2}\tnorm{({\bf v}_h,{\bf q}_h)}_{v,p}^2,
  \end{multline}
  with $\psi > 0$ a constant. Define $0< \epsilon <
  \min(1,\beta_vc_{\alpha_v}^{-2})/4$. Multiplying~\cref{eq:betapqhleYoung}
  by $\epsilon$ and adding to~\cref{eq:betavvhCqhY}, we obtain
  \begin{multline}
    \label{eq:allcombinedterms}
    \beta_v\tnorm{{\bf v}_h}^2_v + \envert{{\bf q}_h}_p^2
    + \epsilon\beta_p^2\norm{q_h}^2_{0,\Omega}
    \le
    4\epsilon c_{\alpha_v}^2\tnorm{ {\bf v}_h }_{v}^2
    + 4\epsilon \envert{{\bf q}_h}^2_p
    + \frac{\psi}{2}\tnorm{({\bf v}_h,{\bf q}_h)}_{v,p}^2
    \\
    + \del{\frac{1}{2\psi} + 2\epsilon}
    \del{\sup_{({\bf w}_h,{\bf r}_h)\in X_h^{\star}}
      \frac{B_h(({\bf v}_h,{\bf q}_h),({\bf w}_h,{\bf r}_h))}
      {\tnorm{({\bf w}_h,{\bf r}_h)}_{v,p}}}^2.
  \end{multline}
  Re-arranging, the result follows with
  \begin{equation}
    \sigma = \del{\frac{
      2 \min \cbr{\beta_v-4\epsilon c_{\alpha_v}^2, 1-4\epsilon,
        \epsilon\beta_p^2} - \psi}{\psi^{-1} + 4\epsilon}}^{1/2}.
  \end{equation}
\end{proof}
\Cref{lem:stab_isfem} proves that the discrete problem is
well-posed. We now show that the bilinear form $B_h$
in~\cref{eq:bilinearformStokes} is bounded.

\begin{lemma}[Boundedness]
  \label{lem:bounded_isfem}
  There exists a constant $C_B > 0$, independent of $h$, such that for
  all $({\bf v},{\bf q})\in X^{\star}(h)$ and $({\bf w}_h,{\bf
    r}_h)\in X_h^{\star}$
  \begin{equation}
    \label{eq:bounded_isfem}
    B_h(({\bf v},{\bf q}),({\bf w}_h,{\bf r}_h)) \le
    C_B \tnorm{({\bf v}, {\bf q})}_{v',p'}\tnorm{({\bf w}_h, {\bf r}_h)}_{v,p}.
  \end{equation}
\end{lemma}
\begin{proof}
  Let $({\bf v},{\bf q})\in X^{\star}(h)$ and~$({\bf w}_h,{\bf r}_h)
  \in X_h^{\star}$. From the definition of~$B_h$,
  \begin{equation}
    B_h(({\bf v},{\bf q}),({\bf w}_h,{\bf r}_h)) =
    a_h({\bf v}, {\bf w}_h)
    + b_h({\bf q}, {\bf w}_h)
    - b_h({\bf r}_h, {\bf v})
    + c_h({\bf q}, {\bf r}_h).
  \end{equation}
  We will bound each of the terms separately. By \cref{lem:bound_ah}
  and the Cauchy--Schwarz inequality,
  \begin{align}
    \label{eq:boundahBh}
    \envert{a_h({\bf v}, {\bf w}_h)}
    &\le C_a\tnorm{({\bf v},{\bf 0})}_{v' ,p'}\tnorm{({\bf w}_h,{\bf 0})}_{v,p}, \\
    \label{eq:boundchBh}
    \envert{c_h({\bf q},{\bf r}_h)}
    &\le c\tnorm{(0,{\bf q})}_{v,p}\tnorm{(0,{\bf r}_h)}_{v,p},
  \end{align}
  with $c > 0$ independent of~$h$. Next we consider a bound for
  $b_h({\bf r}_h, {\bf v})$. From the definition of
  $b_{h}$~\cref{eq:formB}, we note that
  \begin{equation}
    b_h({\bf r}_h, {\bf v}) =
    - \int_{\Omega}r_h \nabla \cdot v\dif x
    + \sum_{K\in\mathcal{T}}\int_{\partial K}(v-\bar{v})\cdot n
    \bar{r}_h \dif s.
  \end{equation}
  It is clear that
  \begin{equation}
    \label{eq:bound_bh_term1}
    \envert{- \int_{\Omega}r_h \nabla \cdot v\dif x}
    \le \sum_{K\in\mathcal{T}}\norm{\nabla v}_{0,K}\norm{r_h}_{0,\Omega}
    \le \tnorm{({\bf v},{\bf 0})}_{v,p}\tnorm{({\bf 0},{\bf r}_h)}_{v,p},
  \end{equation}
  and
  \begin{equation}
    \envert{\sum_{K\in\mathcal{T}}\int_{\partial K}(v - \bar{v})\cdot n
      \bar{r}_h \dif s}
    \le \del{\sum_{K\in\mathcal{T}}h_K^{-1}\norm{v-\bar{v}}_{0,\partial K}^2}^{1/2}
    \del{\sum_{K \in \mathcal{T}} h_K \norm{\bar{r}_h}_{0, \partial K}^2}^{1/2}.
  \end{equation}
  Note, however, that
  \begin{equation}
    h_K\norm{\bar{r}_h}^2_{0,\partial K}
      \le
    2h_K\norm{\bar{r}_h-r_h}^2_{0,\partial K} + c\norm{r_h}_{0,K}^2,
  \end{equation}
  where we have used the discrete trace
  inequality~\cref{eq:trace_inequality} ($c > 0$ is independent
  of~$h$). We therefore find that
  \begin{equation}
    \envert{\sum_{K\in\mathcal{T}}\int_{\partial K}(v-\bar{v})\cdot n
      \bar{r}_h\dif s}
    \le c\alpha_v^{-1/2}\tnorm{({\bf v},{\bf 0})}_{v,p}\tnorm{({\bf 0}, {\bf r}_h)}_{v,p}.
  \end{equation}
  It follows that
  \begin{equation}
    \label{eq:bound_bh_final}
    \envert{b_h({\bf r}_h, {\bf v})}
    \le
    (1 + c\alpha_v^{-1/2})\tnorm{({\bf v},{\bf 0})}_{v,p}\tnorm{({\bf 0}, {\bf r}_h)}_{v,p}.
  \end{equation}
  The final term to bound is $b_h({\bf q}, {\bf w}_h)$. Using the Cauchy--Schwarz inequality,
  \begin{equation}
    \envert{\sum_{K\in\mathcal{T}}\int_{\partial K}(w_h-\bar{w}_h)\cdot n \bar{q}\dif s}
    \le  c\alpha_v^{-1/2} \tnorm{({\bf w}_h,{\bf 0})}_{v,p}
    \tnorm{({\bf 0},{\bf q})}_{v',p'}.
  \end{equation}
  Furthermore, as in~\cref{eq:bound_bh_term1},
  $\envert{-\int_{\Omega}q \nabla \cdot w_h\dif x} \le \tnorm{({\bf
      w}_h, {\bf 0})}_{v,p} \tnorm{({\bf 0},{\bf q})}_{v,p}$.
  It follows that
  \begin{equation}
    \label{eq:bound_bh_final2}
    \envert{b_h({\bf q}, {\bf w}_h)}
    \le
    (1 + c\alpha_v^{-1/2})\tnorm{({\bf w}_h,{\bf 0})}_{v,p}
    \tnorm{({\bf 0}, {\bf q})}_{v',p'}.
  \end{equation}
  Collecting the bounds~\cref{eq:boundahBh}, \cref{eq:boundchBh},
  \cref{eq:bound_bh_final} and~\cref{eq:bound_bh_final2}, the result
  follows with $C_B = c(1+\alpha_v^{-1/2})$, where we have used
  that~$C_a = c(1 + \alpha_v^{-1/2})$.
\end{proof}

\section{Error analysis}
\label{s:errorAnalysis}

With the stability and boundedness results from the preceding section,
we can now develop convergence results for the method.

\subsection{Estimates in mesh-dependent norms}

\begin{theorem}[$\tnorm{\cdot}_{v, p}$-norm error estimate]
  \label{thm:vpnormestimate}
  Let $(u, p) \in X$ be the solution of the Stokes
  problem~\cref{eq:stokes}, and ${\bf u} = (u, \gamma(u))$ and ${\bf
    p} = (p, \gamma(p))$. Let $({\bf u}_h, {\bf p}_h) \in X_h^{\star}$
  solve~\cref{eq:compact_wf}. Then there exists a constant $C_I > 0$,
  independent of~$h$, such that
  \begin{equation}
    \tnorm{({\bf u}-{\bf u}_h, {\bf p}-{\bf p}_h)}_{v,p}
    \le
    C_I\inf_{({\bf v}_h,{\bf q}_h)\in X_h^{\star}}\tnorm{({\bf u}-{\bf v}_h,
      {\bf p} - {\bf q}_h)}_{v',p'}.
  \end{equation}
\end{theorem}

\begin{proof}
  This is a direct consequence of stability (\cref{lem:stab_isfem}),
  consistency (\cref{lem:consistency}), and boundedness of $B_{h}$
  (\cref{lem:bounded_isfem}). The theorem follows
  with~$C_I = 1 + \sigma^{-1}C_B$.
\end{proof}

For the analysis in the remainder of this section we introduce
continuous interpolants. Let $(u, p) \in \sbr{H^{k+1}(\Omega)}^d
\times H^{l+1}(\Omega)$. We denote the standard continuous
interpolant~\cite{Brenner:book} of ${\bf u} = (u, \gamma(u))$ by
$\mathcal{I}^u_h{\bf u} = (\mathcal{I}^u_h u, \bar{\mathcal{I}}^u_h
u)$, hence $\mathcal{I}^u_hu \in V_h \cap C^{0}(\bar{\Omega})$ and
$\bar{\mathcal{I}}^u_h u = \mathcal{I}^u_h u|_{\Gamma^0} \in
\bar{V}_h$.  Furthermore, let $\mathcal{I}_h^p{\bf p} =
(\mathcal{I}^p_h p, \bar{\mathcal{I}}^p_h p)$ for the pressure, where
$\mathcal{I}^p_h p$ is the continuous Scott--Zhang
interpolant~\cite{Scott:1990}, and $\bar{\mathcal{I}}^p_h p =
\mathcal{I}^p_h p|_{\Gamma^0}$. Then $\mathcal{I}^p_h p \in Q_h \cap
C^{0}(\bar{\Omega})$ and $\bar{\mathcal{I}}^p_h p = \mathcal{I}^p_h
p|_{\Gamma^0} \in \bar{Q}_h$.  The interpolation estimates read
\begin{equation}
  \label{eq:interpolationestimate}
  \norm{u - \mathcal{I}_h^u u}_{m, K} \le c h_K^{l + 1 - m} \envert{u}_{l + 1,K}, \quad
  \norm{p - \mathcal{I}_h^p p}_{m, K} \le c h_K^{l + 1 - m} \envert{p}_{l + 1,K}.
\end{equation}

\begin{lemma}[Convergence rate in the $\tnorm{\cdot}_{v,p}$-norm]
  \label{lem:convergencerateintnorm}
  Let $(u, p)\in \sbr{H^{k+1}(\Omega)}^d \times H^{l+1}(\Omega)$ solve
  the Stokes problem~\cref{eq:stokes}, and let
  $({\bf u}_h, {\bf p}_h)\in X_h^{\star}$ solve the finite element
  problem~\cref{eq:compact_wf}. Let ${\bf u} = (u, \gamma(u))$ and
  ${\bf p}=(p, \gamma(p))$.  There exists a constant $C_R > 0$,
  independent of~$h$, such that
  \begin{equation}
    \tnorm{({\bf u}-{\bf u}_h, {\bf p}-{\bf p}_h)}_{v, p}
    \le C_R\del{h^k \norm{u}_{k + 1, \Omega} + h^{l+1} \norm{p}_{l + 1, \Omega}}.
  \end{equation}
\end{lemma}

\begin{proof}
  The proof is similar to that in \cite[Lemma~5.5]{Wells:2011}. Let
  $\mathcal{I}^u_h{\bf u} = (\mathcal{I}^u_h u, \Bar{\mathcal{I}}^u_h
  u)$ and $\mathcal{I}_h^p{\bf p} = (\mathcal{I}^p_h p,
  \Bar{\mathcal{I}}^p_h p)$ be the continuous interpolants of the
  velocity and pressure, respectively.  For the $\tnorm{\cdot}_{v',
    p'}$ norm, we have
  \begin{multline}
    \tnorm{({\bf u} - \mathcal{I}^u_h{\bf u} ,{\bf p}
      - \mathcal{I}^p_h{\bf p})}_{v',p'}^2
    =
    \sum_{K\in\mathcal{T}}\norm{\nabla(u - \mathcal{I}^u_h u)}^2_{0, K}
    \\
    +\sum_{K\in\mathcal{T}}\frac{\alpha_{v}}{h_K}
    \norm{(u - \bar{\mathcal{I}}^u_h u) - (u - \mathcal{I}^u_h u)}^2_{0,\partial K}
    + \sum_{K \in \mathcal{T}}\norm{p - \mathcal{I}^p_h p}^2_{0,K}
    \\
    + \sum_{K\in\mathcal{T}} \alpha_{p} h_K
    \norm{(p-\bar{\mathcal{I}}_h^pp) - (p-\mathcal{I}_h^pp)}^2_{0,\partial K}
    + \sum_{K\in\mathcal{T}} \frac{h_K}{\alpha_v}
    \norm{\frac{\partial u}{\partial n} - \frac{\partial\mathcal{I}^u_hu}{\partial n}}^2_{0,\partial K}
    \\
    + \sum_{K\in\mathcal{T}} h_K\norm{p - \bar{\mathcal{I}}_h^pp}^2_{0,\partial K}.
  \end{multline}
  Using the interpolation estimate~\cref{eq:interpolationestimate},
  \begin{equation}
    \begin{split}
      \norm{\nabla(u-\mathcal{I}^u_h u)}^2_{0,K} &\le ch^{2k} \envert{u}^2_{k+1,K}
      \\
      \frac{\alpha_{v}}{h_K}
      \norm{(u-\bar{\mathcal{I}}^u_h u) - (u-\mathcal{I}^u_h u)}^2_{0,\partial K}
      &= 0
      \\
      \frac{h_K}{\alpha_v}\norm{\frac{\partial u}{\partial n} - \frac{\partial\mathcal{I}^u_hu}{\partial n}}^2_{0,\partial K}
      &\le c\alpha_v^{-1}\del{\envert{u-\mathcal{I}^u_hu}^2_{1,K}
      + h_K^2\envert{u-\mathcal{I}^u_hu}^2_{2,K}}
      \\
      &\le c\alpha_v^{-1}h^{2k} \envert{u}^2_{k+1,K}
      \\
      \norm{p-\mathcal{I}^p_hp}^2_{0,K} &\le ch_K^{2(l+1)} \envert{p}^2_{l+1,K}
      \\
      \alpha_ph_K\norm{(p-\bar{\mathcal{I}}_h^p p) - (p - \mathcal{I}_h^p p)}^2_{0,\partial K} &= 0
      \\
      h_K\norm{p-\bar{\mathcal{I}}_h^pp}^2_{0,\partial K}
      &=
      h_K^2\norm{p-\mathcal{I}_h^pp}^2_{0,\partial K}
      \\
      &\le
      ch_K \del{\norm{p-\mathcal{I}_h^pp}_{0,K}^2
      + h_K^2 \envert{p - \mathcal{I}_h^p p}^2_{1,K}}
      \\
      &\le ch_K^{2l+3} \envert{p}^2_{l,K}.
    \end{split}
  \end{equation}
  It follows that
  \begin{equation}
    \label{eq:uIupIp_equiv}
    \tnorm{({\bf u}-\mathcal{I}^u_h{\bf u},
      {\bf p} - \mathcal{I}^p_h{\bf p})}_{v',p'}
    \le
    c(1+\alpha_v^{-1})^{1/2} \del{h^k \envert{u}_{k+1,\Omega}
    + h^{l + 1} \envert{p}_{l + 1, \Omega}},
  \end{equation}
  and by application of \cref{thm:vpnormestimate} the result follows
  with $C_R = cC_I(1 + \alpha_v^{-1})^{1/2}$.
\end{proof}

\subsection{Estimates in the $L^{2}$ norm}

To find an error estimate in the $L^2$-norm for the velocity, we will
rely on the following regularity assumption. If $f \in
[L^2(\Omega)]^{d}$ for~\cref{eq:stokes} and $(u, p)$ solves Stokes
problem~\cref{eq:stokes}, we have
\begin{equation}
  \label{eq:regularity}
  \norm{u}_{2, \Omega} + \norm{p}_{1, \Omega} \le c_r\norm{f}_{0, \Omega}.
\end{equation}
where $c_{r}$ is a constant. Satisfaction of this regularity estimate
places some restrictions on the shape of the domain~$\Omega$.

\begin{theorem}[Velocity error estimate in the $L^2$-norm]
  \label{thm:L2velocityError}
  Let $(u, p) \in X$ solve the Stokes problem~\cref{eq:stokes}, and
  ${\bf u} = (u, \gamma(u))$ and ${\bf p} = (p, \gamma(p))$, and let
  $({\bf u}_h, {\bf p}_h) \in X_h^{\star}$ be the solution
  to~\cref{eq:compact_wf}.  Subject to the regularity condition
  in~\cref{eq:regularity}, there exists a constant $C_V > 0$,
  independent of~$h$, such that
  \begin{equation}
    \norm{u-u_h}_{0,\Omega}
      \le C_V h \tnorm{({\bf u} - {\bf u}_h, {\bf p} - {\bf p}_h)}_{v^{\prime}, p^{\prime}}.
  \end{equation}
\end{theorem}
\begin{proof}
  Let $(\zeta, \xi) \in X$ solve the Stokes problem~\cref{eq:stokes}
  with $f = (u - u_h) \in [L^2(\Omega)]^d$. From the regularity
  assumption,
  \begin{equation}
    \label{eq:usingregularity}
    \norm{\zeta}_{2,\Omega} + \norm{\xi}_{1,\Omega}
    \le c_r\norm{u - u_h}_{0,\Omega}.
  \end{equation}
  From the definition of $a_h$~\cref{eq:formA},
  \begin{equation}
    \label{eq:ah_u_zeta}
    \begin{split}
      a_h({\bf u} - {\bf u}_h, (\zeta, \gamma(\zeta)))
      &=
      \int_{\Omega}\nabla (u-u_h) : \nabla \zeta\dif x
      - \sum_{K\in\mathcal{T}}\int_{\partial K}(\bar{u}_h-u_h)\cdot \frac{\partial \zeta}{\partial n}\dif s
      \\
      &=
      -\int_{\Omega}(u-u_h) \cdot \Delta \zeta\dif x
      + \sum_{K\in\mathcal{T}}\int_{\partial K}(u - \bar{u}_h)\cdot \frac{\partial \zeta}{\partial n}\dif s
      \\
     &=
      -\int_{\Omega} (u-u_h) \cdot \Delta \zeta\dif x,
    \end{split}
  \end{equation}
  where the second equality is due to integration by parts and the
  third is due to the single valuedness of $u$, $\bar{u}_h$ and
  $\nabla\zeta n$ across element boundaries and $u = \bar{u}_h = 0$
  on~$\Gamma$. By definition of $b_h$ in~\cref{eq:formB}, we also
  find that
  \begin{equation}
    \label{eq:bh_p_xi}
    \begin{split}
      b_h((\xi, \gamma(\xi)), {\bf u}-{\bf u}_h)
      &=
      - \int_{\Omega}\xi \nabla \cdot (u-u_h)\dif x
      + \sum_{K\in\mathcal{T}}\int_{\partial K}(\bar{u}_h-u_h)\cdot n \xi\dif s
      \\
      &=
      \int_{\Omega}\nabla \xi \cdot (u-u_h)\dif x
      - \sum_{K\in\mathcal{T}}\int_{\partial K}(u - \bar{u}_h)\cdot n \xi\dif s
      \\
      &=
      \int_{\Omega}\nabla \xi \cdot (u-u_h)\dif x,
    \end{split}
  \end{equation}
  where the second equality is due to integration by parts and the
  third is due to the single-valuedness of $u$, $\bar{u}_h$ and $\xi$
  across element boundaries and $u = \bar{u}_h = 0$
  on~$\Gamma$. Combining~\cref{eq:ah_u_zeta} and~\cref{eq:bh_p_xi}
  and using that $-\Delta\zeta + \nabla\xi = (u-u_h)$,
  \begin{multline}
    \norm{u-u_h}^2_{0,\Omega}=\int_{\Omega}(u-u_h)^2\dif x
    =\int_{\Omega}(-\Delta\zeta+\nabla\xi)\cdot(u-u_h)\dif x
    \\
    = a_h({\bf u}-{\bf u}_h, (\zeta, \gamma(\zeta))) + b_h((\xi, \gamma(\xi)), {\bf
      u}-{\bf u}_h).
  \end{multline}
  Furthermore, again by definition of~$b_h$~\cref{eq:formB},
  \begin{equation}
    \begin{split}
      b_h({\bf p}-{\bf p}_h, (\zeta,\zeta)) = &
      - \int_{\Omega}(p-p_h) \nabla \cdot \zeta\dif x = 0,
    \end{split}
  \end{equation}
  due to $\nabla \cdot \zeta = 0$. From the definition of
  $c_h$~\cref{eq:formC}, it is clear that $c_h({\bf p} - {\bf
    p}_h,(\zeta,\zeta))=0$. It therefore follows that
  \begin{equation}
    \norm{u-u_h}^2_{0,\Omega}
    = B_h(({\bf u}-{\bf u}_h,{\bf p}-{\bf p}_h), ((\zeta,\gamma(\zeta)),(\xi, \gamma(\xi)))).
  \end{equation}
  Using consistency (\cref{lem:consistency}), boundedness of
  $B_h$ on $X^{\star}(h) \times
  X^{\star}(h)$\footnote{\cref{lem:bounded_isfem} provides
    boundedness of $B_h$ on $X^{\star}(h) \times X_h^{\star}$ but the
    proof for boundedness on $X^{\star}(h) \times X^{\star}(h)$ is
    similar with both norms on the right-hand side of
    \cref{eq:bounded_isfem} being $\tnorm{\cdot}_{v^{\prime}, p^{\prime}}$}, we find
  \begin{equation}
    \label{eq:bounduuhalmost}
    \begin{split}
      \norm{u-u_h}^2_{0,\Omega}
      &=
      B_h(({\bf u}-{\bf u}_h,{\bf p}-{\bf p}_h),((\zeta-\mathcal{I}_h^u\zeta,\zeta-\bar{\mathcal{I}}_h^u\zeta), (\xi-\mathcal{I}_h^p\xi,\xi-\bar{\mathcal{I}}_h^p\xi)))
      \\
      &\le C_B\tnorm{({\bf u}-{\bf u}_h,{\bf p}-{\bf p}_h)}_{v',p'}
      \tnorm{((\zeta-\mathcal{I}_h^u\zeta,\zeta-\bar{\mathcal{I}}_h^u\zeta),(\xi-\mathcal{I}_h^p\xi,\xi-\bar{\mathcal{I}}_h^p\xi))}_{v',p'}.
    \end{split}
  \end{equation}
  Using the interpolation estimate~\cref{eq:uIupIp_equiv}, we note
  that
  \begin{equation}
    \label{eq:usinginterpolationestimate}
    \tnorm{((\zeta-\mathcal{I}_h^u \zeta, \zeta - \bar{\mathcal{I}}_h^u\zeta),
      (\xi - \mathcal{I}_h^p \xi, \xi - \bar{\mathcal{I}}_h^p\xi))}_{v',p'}
    \le
    c(1 + \alpha_v^{-1})^{1/2}h \del{\envert{\zeta}_{2, \Omega}
      + \envert{\xi}_{1,\Omega}}.
  \end{equation}
  It follows from~\cref{eq:bounduuhalmost}
  and~\cref{eq:usinginterpolationestimate},
  \begin{equation}
    \begin{split}
      \norm{u-u_h}^2_{0,\Omega}
      \le &c(1+\alpha_v^{-1})^{1/2}C_B
      h\tnorm{({\bf u}-{\bf u}_h,{\bf p} - {\bf p}_h)}_{v',p'}
      \del{\norm{\zeta}_{2,\Omega} + \norm{\xi}_{1,\Omega}}.
    \end{split}
  \end{equation}
  Using the regularity estimate~\cref{eq:usingregularity}, we obtain
  \begin{equation}
    \norm{u-u_h}^2_{0,\Omega} \le c(1+\alpha_v^{-1})^{1/2}C_Bc_r
    h\tnorm{({\bf u}-{\bf u}_h,{\bf p}-{\bf p}_h)}_{v',p'}
    \norm{u-u_h}_{0,\Omega},
  \end{equation}
  from which the theorem follows with
  $C_V=c(1+\alpha_v^{-1})^{1/2}C_Bc_r$.
\end{proof}

We can now obtain a convergence rate for the velocity error in the
$L^2$-norm.
\begin{lemma}[Convergence rate for the velocity in the $L^2$-norm]
  \label{lem:l2convrate_v}
  Let $(u, p) \in \sbr{H^{k+1}(\Omega)}^d \times H^{l + 1}(\Omega)$ be
  the solution of the Stokes problem~\cref{eq:stokes}, and ${\bf u} =
  (u, \gamma(u))$ and ${\bf p} = (p, \gamma(p))$, and let $({\bf u}_h,
  {\bf p}_h)\in X_h^{\star}$ solve~\cref{eq:compact_wf}. Subject to
  the regularity condition~\cref{eq:regularity}, there exists a
  constant $C_C > 0$, independent of~$h$, such that
  \begin{equation}
    \norm{u-u_h}_{0,\Omega}
    \le C_C \del{h^{k + 1} \norm{u}_{k + 1,\Omega}
      + h^{l + 2} \norm{p}_{l+1, \Omega}}.
  \end{equation}
\end{lemma}
\begin{proof}
  We first show that $\tnorm{({\bf v}_h, {\bf q}_h)}_{v, p}$ and
  $\tnorm{({\bf v}_h, {\bf q}_h)}_{v^{\prime}, p^{\prime}}$ are equivalent norms on
  $X_h^{\star}$,~i.e.,
  \begin{equation}
    \tnorm{({\bf v}_h,{\bf q}_h)}^2_{v,p}
    \le \tnorm{({\bf v}_h, {\bf q}_h)}^2_{v^{\prime}, p^{\prime}}
    \le c \tnorm{({\bf v}_h,{\bf q}_h)}^2_{v,p}.
  \end{equation}
  The first inequality is trivial. The second inequality follows from
  \begin{equation}
    \sum_{K\in\mathcal{T}}h_K\norm{\bar{q}_h}^2_{0,\partial K} \le
    \sum_{K\in\mathcal{T}}h_K \del{\norm{\bar{q}_h - q_h}^2_{0,\partial K} +
      \norm{q_h}_{0,\partial K}^2}
    \le \envert{{\bf q}_h}_p^2 + C_t^2\sum_{K\in\mathcal{T}}\norm{q_h}^2_{0, K}
  \end{equation}
  where the last inequality above is due to the discrete trace
  inequality~\cref{eq:trace_inequality}. From the equivalence of
  $\tnorm{({\bf v}_h, {\bf q}_h)}_{v,p}$ and $\tnorm{({\bf v}_h, {\bf
      q}_h)}_{v^{\prime}, p^{\prime}}$ on~$X_h^{\star}$, we find by
  \cref{lem:stab_isfem} that for all~$({\bf v}_h, {\bf q}_h) \in
  X_h^{\star}$
  \begin{equation}
    \label{eq:infsupstability-star}
    c \sigma \tnorm{({\bf v}_h,{\bf q}_h)}_{v^{\prime}, p^{\prime}} \le
    \sup_{({\bf w}_h,{\bf r}_h)\in X_h^{\star}}
    \frac{B_h(({\bf v}_h,{\bf q}_h),({\bf w}_h,{\bf r}_h))}
    {\tnorm{({\bf w}_h,{\bf r}_h)}_{v^{\prime}, p^{\prime}}}.
  \end{equation}
  By \cref{lem:bounded_isfem} we have boundedness of $B_h$ on
  $X^{\star}(h) \times X_h^{\star}$ with respect to the
  $\tnorm{(\cdot, \cdot)}_{v^{\prime}, p^{\prime}}$ and
  $\tnorm{(\cdot, \cdot)}_{v,p}$ norms. The bilinear form $B_h$,
  however, is also bounded on $X^{\star}(h) \times X^{\star}(h)$, but
  with respect to only the $\tnorm{(\cdot, \cdot)}_{v^{\prime},
    p^{\prime}}$ norm:
  \begin{equation}
    \label{eq:bounded_isfem_staronly}
    B_h(({\bf v},{\bf q}),({\bf w},{\bf r})) \le
    C_B \tnorm{({\bf v}, {\bf q})}_{v^{\prime}, p^{\prime}}\tnorm{({\bf w}, {\bf r})}_{v^{\prime}, p^{\prime}}.
  \end{equation}
  Using~\cref{eq:infsupstability-star} and consistency
  (\cref{lem:consistency}), we find that
  \begin{equation}
    \begin{split}
      c\sigma \tnorm{({\bf u}_h-{\bf v}_h, {\bf p}_h-{\bf q}_h)}_{v' ,p'}
      &\le
      \sup_{({\bf w}_h,{\bf r}_h)\in X_h^{\star}}
      \frac{B_h(({\bf u}_h-{\bf v}_h,{\bf p}_h-{\bf q}_h),({\bf w}_h,{\bf r}_h))}
           {\tnorm{({\bf w}_h,{\bf r}_h)}_{v^{\prime}, p^{\prime}}}
           \\
      &=
      \sup_{({\bf w}_h,{\bf r}_h)\in X_h^{\star}}
      \frac{B_h(({\bf u}-{\bf v}_h,{\bf p}-{\bf q}_h),({\bf w}_h,{\bf r}_h))}
      {\tnorm{({\bf w}_h,{\bf r}_h)}_{v^{\prime}, p^{\prime}}}.
    \end{split}
  \end{equation}
  Boundedness of $B_h$~\cref{eq:bounded_isfem_staronly} results in
  \begin{equation}
    \frac{c\sigma}{C_B} \tnorm{({\bf u}_h-{\bf v}_h, {\bf p}_h-{\bf q}_h)}_{v^{\prime}, p^{\prime}}
    \le
    \tnorm{({\bf u}-{\bf v}_h,{\bf p}-{\bf q}_h)}_{v^{\prime}, p^{\prime}},
  \end{equation}
  and by the triangle inequality (similar to
  \cref{thm:vpnormestimate}), we find
  \begin{equation}
    \label{eq:L2vnormestimate}
    \tnorm{({\bf u}-{\bf u}_h, {\bf p}-{\bf p}_h)}_{v^{\prime}, p^{\prime}} \le
    \del{1 + \frac{cC_B}{\sigma}}
    \inf_{({\bf v}_h,{\bf q}_h)\in X_h^{\star}}
    \tnorm{({\bf u}-{\bf v}_h, {\bf p}-{\bf q}_h)}_{v^{\prime}, p^{\prime}}.
  \end{equation}
  By \cref{thm:L2velocityError} and~\cref{eq:L2vnormestimate} we
  therefore find
  \begin{equation}
    \begin{split}
      \norm{u-u_h}_{0,\Omega} &\le C_V h \tnorm{({\bf u}-{\bf u}_h,
        {\bf p} - {\bf p}_h)}_{v^{\prime}, p^{\prime}}
      \\
      &\le
      C_V \del{1 + \frac{cC_B}{\sigma}} h
      \inf_{({\bf v}_h, {\bf q}_h) \in X_h^{\star}}\tnorm{({\bf u} - {\bf v}_h,
        {\bf p} - {\bf q}_h)}_{v^{\prime}, p^{\prime}}.
    \end{split}
  \end{equation}
  In particular
  \begin{equation}
    \norm{u-u_h}_{0,\Omega}
    \le
    C_V \del{1 + \frac{cC_B}{\sigma}}
    h \tnorm{({\bf u} - \mathcal{I}_h^u{\bf u},
      {\bf p} - \mathcal{I}_h^p{\bf p})}_{v^{\prime}, p^{\prime}}.
  \end{equation}
  Using the interpolation estimate~\cref{eq:uIupIp_equiv}, we obtain
  \begin{equation}
    \norm{u - u_h}_{0, \Omega}
    \le
    c(1+\alpha_v^{-1})^{1/2}C_V \del{1 + \frac{cC_B}{\sigma}}
    \del{h^{k + 1} \envert{u}_{k + 1, \Omega} + h^{l + 2} \envert{p}_{l + 1, \Omega}},
  \end{equation}
  and the result follows with $C_C = c(1 + \alpha_v^{-1})^{1/2} C_V(1
  + cC_B \sigma^{-1})$.
\end{proof}

We end this section with the convergence rate for the pressure in
the~$L^2$-norm.

\begin{lemma}[Convergence rate for the pressure in the $L^2$-norm]
  \label{lem:l2convrate_p}
  Let $(u, p)\in \sbr{H^{k+1}(\Omega)}^d \times H^{l+1}(\Omega)$ solve
  the Stokes problem~\cref{eq:stokes}, and ${\bf u} = (u, \gamma(u))$
  and ${\bf p} = (p, \gamma(p))$, and let
  $({\bf u}_h, {\bf p}_h) \in X_h^{\star}$
  solve~\cref{eq:compact_wf}. Let $C_R$ be the constant defined in
  \cref{lem:convergencerateintnorm}. Subject to the regularity
  condition in~\cref{eq:regularity}, the following inequality holds,
  \begin{equation}
    \norm{p-p_h}_{0,\Omega}
    \le C_R \del{h^k \norm{u}_{k + 1,\Omega} + h^{l + 1}\norm{p}_{l + 1,\Omega}}.
  \end{equation}
\end{lemma}

\begin{proof}
  Since the $L^2$-norm of the pressure is part of the norm $\tnorm{(\cdot,
  \cdot)}_{v, p}$ we note that
  \begin{equation}
    \norm{p - p_h}_{0, \Omega}
    \le \tnorm{({\bf u} - {\bf u}_h, {\bf p} - {\bf p}_h)}_{v, p}
    \le C_R \del{h^k \norm{u}_{k + 1,\Omega} + h^{l + 1} \norm{p}_{l + 1,\Omega}},
  \end{equation}
  where the last inequality is due to
  \cref{lem:convergencerateintnorm}.
\end{proof}

\Cref{lem:l2convrate_v,lem:l2convrate_p} show that if
$\sbr{P_k}^d$--$P_k$ or $\sbr{P_k}^d$--$P_{k-1}$ elements are used for
the velocity-pressure approximation, then
\begin{equation}
  \norm{u-u_h}_{0,\Omega} \le c h^{k+1} \quad \text{and} \quad
  \norm{p-p_h}_{0,\Omega} \le c h^{k}.
\end{equation}
For the $\sbr{P_k}^d$--$P_k$ element we therefore find a sub-optimal
error estimate, while an optimal estimate is obtained for the
$\sbr{P_k}^d$--$P_{k-1}$ element.

The \emph{a priori} error estimates are consistent with the
experimentally observed convergence rates in~\cite{Labeur:2012} for
the case of $C^{0}$-conforming facet functions.

\section{Conclusions}
\label{sec:conclusions}

We have analysed a hydridized DG/interface stabilized method for the
Stokes equations, and proven inf-sup stability and optimal convergence
rates.  The developed convergence estimates are consistent with the
experimentally observed convergence rates presented in earlier
publications. The method is particularly appealing as it can be
constructed to have the same number of global degrees of freedom and
the same global matrix operator structure as a conforming formulation,
yet it is locally conservative. Moreover, on simplices the local
velocity field can be point-wise divergence-free. These properties
make the method an excellent candidate for coupling to transport
equations. When extended to the incompressible Navier--Stokes
equations, the structure of the method makes the incorporation of
standard DG-type stabilization of the advective terms straightforward.

\subsubsection*{Acknowledgements}

SR gratefully acknowledges support from the Natural Sciences and
Engineering Research Council of Canada through the Discovery Grant
program (RGPIN-05606-2015) and the Discovery Accelerator Supplement
(RGPAS-478018-2015).

\bibliographystyle{siamplain}
\bibliography{references}

\begin{thebibliography}{10}

\bibitem{Arnold:2002}
{\sc D.~N. Arnold, F.~Brezzi, B.~Cockburn, and L.~D. Marini}, {\em Unified
  analysis of discontinuous {G}alerkin methods for elliptic problems}, SIAM J.
  Numer. Anal., 39 (2002), pp.~1749--1779,
  \url{http://dx.doi.org/10.1137/S0036142901384162}.

\bibitem{Boffi:book}
{\sc D.~Boffi, F.~Brezzi, and M.~Fortin}, {\em Mixed Finite Element Methods and
  Applications}, vol.~44 of Springer Series in Computational Mathematics,
  Springer--Verlag Berlin Heidelberg, 2013.

\bibitem{Brenner:book}
{\sc S.~C. Brenner and L.~R. Scott}, {\em The mathematical theory of finite
  element methods}, vol.~15 of Texts in Applied Mathematics, Springer--Verlag
  New York, second~ed., 2002.

\bibitem{Brezzi:1985}
{\sc F.~Brezzi, J.~J.~D.~Douglas, and L.~D. Marini}, {\em Two families of mixed
  finite elements for second order elliptic problems}, Numer. Math., 47 (1985),
  pp.~217--235, \url{http://dx.doi.org/10.1007/BF01389710}.

\bibitem{Cockburn:2009a}
{\sc B.~Cockburn, J.~Gopalakrishnan, and R.~Lazarov}, {\em Unified
  hybridization of discontinuous {G}alerkin, mixed, and continuous {G}alerkin
  methods for second order elliptic problems}, SIAM J. Numer. Anal., 47 (2009),
  pp.~1319--1365, \url{http://dx.doi.org/10.1137/070706616}.

\bibitem{Cockburn:2011}
{\sc B.~Cockburn, J.~Gopalakrishnan, N.~C. Nguyen, J.~Peraire, and F.~J.
  Sayas}, {\em Analysis of {HDG} methods for {S}tokes flow}, Math. Comp., 80
  (2011), pp.~723--760,
  \url{http://dx.doi.org/10.1090/S0025-5718-2010-02410-X}.

\bibitem{Cockburn:2009}
{\sc B.~Cockburn, J.~Guzm\'an, S.-C. Soon, and H.~K. Stolarski}, {\em An
  analysis of the embedded discontinuous {G}alerkin method for second-order
  elliptic problems}, SIAM J. Numer. Anal., 47 (2009), pp.~2686--2707,
  \url{http://dx.doi.org/10.1137/080726914}.

\bibitem{Cockburn:2014b}
{\sc B.~Cockburn and K.~Shi}, {\em Devising {HDG} methods for {S}tokes flow:
  {A}n overview}, Comput. Fluids, 98 (2014), pp.~221--229,
  \url{http://dx.doi.org/10.1016/j.compfluid.2013.11.017}.

\bibitem{Pietro:book}
{\sc D.~A. Di~Pietro and A.~Ern}, {\em Mathematical Aspects of Discontinuous
  {G}alerkin Methods}, vol.~69 of Math\'ematiques et Applications,
  Springer--Verlag Berlin Heidelberg, 2012.

\bibitem{Hansbo:2002}
{\sc P.~Hansbo and M.~G. Larson}, {\em Discontinuous {G}alerkin methods for
  incompressible and nearly incompressible elasticity by {N}itsche's method},
  Comput. Methods Appl. Mech. Engrg., 191 (2002), pp.~1895--1908,
  \url{http://dx.doi.org/10.1016/S0045-7825(01)00358-9}.

\bibitem{labeur:2007}
{\sc R.~J. Labeur and G.~N. Wells}, {\em A {G}alerkin interface stabilisation
  method for the advection--diffusion and incompressible {N}avier--{S}tokes
  equations}, Comput. Methods Appl. Mech. Engrg., 196 (2007), pp.~4985--5000,
  \url{http://dx.doi.org/10.1016/j.cma.2007.06.025}.

\bibitem{Labeur:2012}
{\sc R.~J. Labeur and G.~N. Wells}, {\em Energy stable and momentum conserving
  hybrid finite element method for the incompressible {N}avier--{S}tokes
  equations}, SIAM J. Sci. Comput., 34 (2012), pp.~A889--A913,
  \url{http://dx.doi.org/10.1137/100818583}.

\bibitem{schotzau:2002}
{\sc D.~Sch\"{o}tzau, C.~Schwab, and A.~Toselli}, {\em Mixed hp-{DGFEM} for
  incompressible flows}, SIAM J. Numer. Anal., 40 (2002), pp.~2171--2194,
  \url{http://dx.doi.org/10.1137/S0036142901399124}.

\bibitem{Scott:1990}
{\sc L.~R. Scott and S.~Zhang}, {\em Finite element interpolation of nonsmooth
  functions satisfying boundary conditions}, Math. Comp., 54 (1990),
  pp.~483--493, \url{https://doi.org/10.1090/S0025-5718-1990-1011446-7}.

\bibitem{Wells:2011}
{\sc G.~N. Wells}, {\em Analysis of an interface stabilized finite element
  method: the advection-diffusion-reaction equation}, SIAM J. Numer. Anal., 49
  (2011), pp.~87--109, \url{http://dx.doi.org/10.1137/090775464}.

\end{thebibliography}
\end{document}